\documentclass{article}

\usepackage{amsfonts}

\usepackage{amsthm}

\usepackage{amssymb}

\usepackage{amsmath}

\usepackage{enumerate}

\usepackage[all]{xy}

\usepackage{graphicx}

\newcommand{\N}{\mathbb{N}}

\newcommand{\R}{\mathbb{R}}

\newcommand{\C}{\mathbb{C}}

\newcommand{\Ga}{\Gamma}

\newcommand{\sw}{\mathcal{S}}

\newcommand{\wgn}{\widetilde{G_n}}

\newcommand{\ahn}{\mathcal{A}(\mathcal{V}_n)}

\newcommand{\V}{\mathcal{V}}

\newcommand{\Cliff}{\text{Cliff}}

\newcommand{\Hi}{\mathcal{H}_0}

\newcommand{\umtwist}{\prod^{U,max}C^*(P_R(\wgn);\mathcal{A}(\mathcal{V}_n))^{\Ga_n}}

\newcommand{\umk}{\prod^{U,max}C^*(P_R(\wgn);\mathcal{K}(\mathcal{L}^2_n))^{\Ga_n}}

\newcommand{\umlk}{\prod^{U,L,max}C^*(P_R(\wgn);\mathcal{K}(\mathcal{L}^2_n))^{\Ga_n}}

\newcommand{\ulmtwist}{\prod^{U,L,max}C^*(P_R(\wgn);\mathcal{A}(\mathcal{V}_n))^{\Ga_n}}

\newcommand{\uatwist}{\prod^U\C[P_R(\wgn);\ahn]^{\Ga_n}}

\newcommand{\ualtwist}{\prod^{U,L}\C[P_R(\wgn);\ahn]^{\Ga_n}}

\theoremstyle{plain} 

\theoremstyle{plain} 

\theoremstyle{plain} 

\theoremstyle{plain}

\theoremstyle{plain} \newtheorem{maxthe}{Theorem}[section] 

\theoremstyle{remark} \newtheorem{maxgen}[maxthe]{Remark}

\theoremstyle{definition} \newtheorem{uniprod}[maxthe]{Definition}

\theoremstyle{plain} \newtheorem{uniiso}[maxthe]{Theorem}

\theoremstyle{plain} \newtheorem{unises}[maxthe]{Lemma}

\theoremstyle{definition} \newtheorem{locdef}[maxthe]{Definition}

\theoremstyle{plain} \newtheorem{lociso}[maxthe]{Proposition}

\theoremstyle{plain} \newtheorem{locthe}[maxthe]{Theorem}

\theoremstyle{plain} \newtheorem{mvthe}[maxthe]{Theorem}

\theoremstyle{definition} \newtheorem{c*h}[maxthe]{Definition}

\theoremstyle{definition} \newtheorem{suppv}[maxthe]{Definition}

\theoremstyle{definition} \newtheorem{dubyas}[maxthe]{Definition}

\theoremstyle{plain} 

\theoremstyle{definition} \newtheorem{twisted}[maxthe]{Definition}

\theoremstyle{definition} \newtheorem{unitwist}[maxthe]{Definition}

\theoremstyle{plain} \newtheorem{max=red}[maxthe]{Proposition}

\theoremstyle{definition} \newtheorem{supp}[maxthe]{Definition}

\theoremstyle{definition} \newtheorem{restwist}[maxthe]{Definition}

\theoremstyle{definition} \newtheorem{basicp}[maxthe]{Definition}

\theoremstyle{plain} \newtheorem{basiclem}[maxthe]{Lemma}

\theoremstyle{definition} \newtheorem{lipdef}[maxthe]{Definition}

\theoremstyle{definition} \newtheorem{lipisom}[maxthe]{Definition}

\theoremstyle{definition} \newtheorem{liphom}[maxthe]{Definition}

\theoremstyle{plain} \newtheorem{liplem}[maxthe]{Lemma}

\theoremstyle{plain} \newtheorem{simlem}[maxthe]{Lemma}

\theoremstyle{plain} \newtheorem{isolem}[maxthe]{Lemma}

\theoremstyle{plain} \newtheorem{insum}[maxthe]{Lemma}

\theoremstyle{plain} \newtheorem{mvcor}[maxthe]{Corollary}

\theoremstyle{plain} \newtheorem{splitlem}[maxthe]{Lemma}

\theoremstyle{plain} \newtheorem{dddthe}{Theorem}[section]

\theoremstyle{plain} \newtheorem{alphalem}[dddthe]{Lemma}

\theoremstyle{plain} \newtheorem{betalem}[dddthe]{Lemma}

\theoremstyle{plain} \newtheorem{gammalem}[dddthe]{Lemma}

\theoremstyle{plain} \newtheorem{gammaiso}[dddthe]{Lemma}

\theoremstyle{plain} \newtheorem{complem}[dddthe]{Lemma}

\theoremstyle{definition} \newtheorem{geot}{Definition}[section]

\theoremstyle{plain} \newtheorem{geotex}[geot]{Lemma}

\theoremstyle{plain} \newtheorem{geotthe}[geot]{Theorem}

\theoremstyle{plain} \newtheorem{geotc}[geot]{Corollary}

\theoremstyle{remark} \newtheorem{gtprob}[geot]{Problem}

\theoremstyle{remark} \newtheorem{gtprob2}[geot]{Problem}

\theoremstyle{remark} \newtheorem{gtprob3}[geot]{Problem}

\title{Higher index theory for certain expanders and Gromov monster groups II}

\author{Rufus Willett and Guoliang Yu\footnote{Partially supported by the NSF.}}

\date{}

\begin{document}

\maketitle

\begin{abstract}
In this paper, the second of a series of two, we continue the study of higher index theory for expanders.  We prove that if a sequence of graphs has girth tending to infinity, then the maximal coarse Baum-Connes assembly map is an isomorphism for the associated metric space $X$. As discussed in the first paper in this series, this has applications to the Baum-Connes conjecture for `Gromov monster' groups. 

We also introduce a new property, `geometric property (T)'.  For the metric space associated to a sequence of graphs, this property is an obstruction to the maximal coarse assembly map being an isomorphism.  This enables us to distinguish between expanders with girth tending to infinity, and, for example, those constructed from property (T) groups.
\end{abstract}

\tableofcontents

\section{Introduction}
 
This paper is the second of a series of two in which we study higher index theory, in particular, the coarse Baum-Connes conjecture, for spaces of graph with large girth.  The reader can find a complete introduction and definitions and conventions in the first paper in this series, \cite{Willett:2010ud}, particularly Sections 1, 2, 3, and 4.

The main purpose of this second paper is to prove the following theorem.

\begin{maxthe}\label{maxthe}
Assume that $X=\sqcup G_n$ is a space of graphs with large girth and bounded geometry.  Then the maximal coarse assembly map
$$
\mu:\lim_{R\to\infty}K_*(P_R(X))\to K_*(C^*_{max}(X))
$$
is an isomorphism.
\end{maxthe}

This implies that the maximal coarse Baum-Connes assembly map is an isomorphism for a certain class of expanders; note that we do not need to assume that this sequence of graphs arises as a sequence of quotients of a property ($\tau$) group.  It also has corollaries for the Baum-Connes assembly map with coefficients for Gromov monster groups: see Section 8 in the first paper in this series \cite{Willett:2010ud}.

The remaining part of the paper is taken up with the introduction of a new property, which we call \emph{geometric property (T)}.  Just as property (T) is an obstruction to the maximal Baum-Connes assembly map being an isomorphism (at least without contradicting the Baum-Connes conjecture), geometric property (T) for a space of graphs is an obstruction to the maximal coarse assembly map being an isomorphism. For a space of graphs, the property `geometric property (T)' is strictly stronger than the property `being an expander'.  Moreover, a Margulis-type expander (i.e.\ an expander built as a sequence of quotients of a discrete group $\Ga$) has geometric property (T) if and only if the original group $\Ga$ has property (T).  This new property suggests some interesting questions -- see Section \ref{geotsec} below.

\subsection*{Outline of the piece} 

Section \ref{backsec} lists some of the notation from the first paper in this series \cite{Willett:2010ud}.

Section \ref{maxsec} outlines the proof of Theorem \ref{maxthe} and describes the main ingredient: a version of the maximal Baum-Connes conjecture with uniform propagation control over an infinite sequence of free groups; its proof is based on that of the Baum-Connes conjecture for a-T-menable groups, due to Higson--Kasparov \cite{Higson:2001eb}, and that of the second author of the coarse Baum-Connes conjecture for metric spaces that coarsely embed in Hilbert space \cite{Yu:200ve}.  Section \ref{locsec} replaces the central statement with one involving \emph{localization algebras} \cite{Yu:1997kb}, which make the propagation control that we need somewhat easier to manage.  Section \ref{twistsec} introduces twisted versions of the (equivariant) localization algebras and Roe algebras, and proves a version of our main statement for these twisted algebras.  Section \ref{dddsec} completes the proof by using the Dirac-dual-Dirac method in infinite dimensions of Higson--Kasparov--Trout \cite{Higson:1999be} and Higson--Kasparov \cite{Higson:2001eb} to reduce the main statement to the twisted case.
 
Section \ref{geotsec} (which is rather less technical than the rest of the paper) introduces \emph{geometric property (T)}.  Spaces of graphs with this property generalize the class of Margulis-type expanders built from property (T) groups.  As the maximal coarse Baum-Connes conjecture always fails for such expanders, our results imply that they form a completely distinct class from those expanders with large girth.  We also include some open problems for this class of expanders, which seem to merit further study.

\section{Notation}\label{backsec}
For the reader's convenience, in this section we record the notation from \cite{Willett:2010ud} that we will need, and provide references to definitions. 

\begin{itemize}
\item $\mathcal{H}_0$.  A fixed infinite dimensional, separable, complex Hilbert space.
\item $\mathcal{K}$.  A copy of the $C^*$-algebra of compact operators on $\mathcal{H}_0$.
\item $X=\sqcup G_n$.  A space of graphs \cite[Definition 1.1]{Willett:2010ud}.
\item $\text{prop}(T)$.  The propagation of an operator \cite[Definition 3.2]{Willett:2010ud}. 
\item $\C[X]$, respetviely $C^*_{max}(X)$.  The algebraic Roe algebra, respectively maximal Roe algebra, of a bounded geometry metric space $X$  \cite[Definition 3.2]{Willett:2010ud}.
\item $\C[X]^\Ga$, respetviely $C^*_{max}(X)^\Ga$.  The equivariant algebraic Roe algebra, respectively maximal equivariant Roe algebra, of a bounded geometry metric space $X$ equipped with a free and proper isometric action of a discrete group $\Ga$ \cite[Definition 3.6]{Willett:2010ud}.
\item $P_R(X)$.  The Rips complex of a uniformly discrete metric space $X$ at scale $R$ \cite[Definition 4.3]{Willett:2010ud}.
\item $\mu:\displaystyle{\lim_{R\to\infty}K_*(P_R(X))\to K_*(C^*_{max}(X)}$.  The maximal coarse assembly map associated to a bounded geometry uniformly discrete metric space $X$ \cite[Section 4]{Willett:2010ud}.
\item $\mu_\Ga:\displaystyle{\lim_{R\to\infty}K_*^\Ga(P_R(X))\to K_*(C^*_{max}(X)^\Ga}$. The maximal Baum-Connes assembly map associated to a bounded geometry uniformly discrete metric space $X$ \cite[Section 4]{Willett:2010ud}.
\item $\Delta\in \C[X]$.  The Laplacian on a space of graphs $X$  \cite[Examples 5.3 (i)]{Willett:2010ud}.
\end{itemize}

\section{Strategy for the proof}\label{maxsec}

In this section we outline our strategy for proving Theorem \ref{maxthe} above.  The essential idea is to reduce to a version of the Baum-Connes conjecture for the sequence of universal covers $(\wgn)_{n\in\N}$ of the sequence $(G_n)$ of finite graphs acted on by the sequence of covering groups $(\Ga_n)_{n\in\N}$.  Some of the ideas in this section are based on the work of Oyono-Oyono and the second author \cite{Oyono-Oyono:2009ua}.

\begin{maxgen}\label{maxgen}
The methods used to prove Theorem \ref{maxthe} above could prove a somewhat more general result.  Indeed, let $Y=\sqcup_{n\in\N} Y_n$ be a disjoint union of finite metric spaces, metrized in a similar way to a space of graphs. The essential ingredient is an asymptotically faithful sequence of covers $(\widetilde{Y_n})_{n\in\N}$ of the $Y_n$ such that the covers are of uniformly bounded geometry, and admit equivariant (for the group of deck transformations) coarse embeddings into Hilbert space with uniform distortion control (i.e.\ with the same $\rho_\pm$ as in \cite[Definition 3.4]{Willett:2010ud}).  For example, each $Y_n$ could be a closed manifold of sectional curvature $-1$, such that the dimension of $Y=\sqcup Y_n$ is finite and so that the covering maps from $\widetilde{Y_n}\to Y_n$ from hyperbolic space are asymptotically faithful.  As Theorem \ref{maxthe} seems to cover the most interesting case, however, and to prevent the notation getting out of control, we focus only on spaces as in the statement of Theorem \ref{maxthe}.
\end{maxgen}

Throughout the rest of this section we fix a space of graphs $X=\sqcup G_n$ as in the statement of Theorem \ref{maxthe}, and let $\wgn$ denote the universal cover of $G_n$ and $\Ga_n$ the associated covering group with respect to some fixed choice of basepoint.

The following definition is a maximal version of \cite[Definition 4.3]{Guentner:2008gd}.  

\begin{uniprod}\label{uniprod}
The \emph{algebraic uniform product} of the algebras $\C[\wgn]^{\Ga_n}$ is the $*$-subalgebra of $\prod_n\C[\wgn]^{\Ga_n}$ consisting of sequences $\mathbf{T}=(T^{(0)},T^{(1)},...)$ such that
\begin{enumerate}[(i)]
\item $\sup_{n}\sup_{x,y\in \wgn}\{\|T^{(n)}_{x,y}\|_\mathcal{K}\}$ is finite;
\item $\sup_n\text{prop}(T^{(n)})$ is finite.
\end{enumerate}
Denote this $*$-algebra by $\prod^U\C[\wgn]^{\Ga_n}$.  The \emph{maximal uniform product} of the algebras $C^*_{max}(\wgn)^{\Ga_n}$, denoted
$$
\prod^{U,max}_nC^*(\wgn)^{\Ga_n},
$$
is the completion of $\prod^U\C[\wgn]^{\Ga_n}$ for the norm
$$
\|\mathbf{T}\|_{max}=\sup_n\{\|\pi(\mathbf{T})\|_{\mathcal{B}(\mathcal{H})}~|~\pi:\prod^U\C[\wgn]^{\Ga_n}\to\mathcal{B}(\mathcal{H}) \text{ a $*$-representation}\}
$$
(it is not too hard to use the fact that the spaces $\wgn$ have uniform bounded geometry to show that this norm is finite).
\end{uniprod}
\noindent
Note that the maximal norm on $\prod^U\C[\wgn]^{\Ga_n}$ as defined above does not seem to be the same as the norm $\|(T^{(0}),T^{(1)},...)\|:=\sup_n\|T^{(n)}\|_{max}$; the notation `$\prod^{U,max}C^*(\wgn)^{\Ga_n}$' rather than `$\prod^UC^*_{max}(\wgn)^{\Ga_n}$' is used for this reason.

Note that there is a \emph{uniform assembly map}
\begin{equation}\label{uniassem}
\mu_U:\lim_{R\to\infty}\prod_nK_*^{\Ga_n}(P_R(\wgn))\to K_*(\prod^{U,max}_n C^*(\wgn)^{\Ga_n}),
\end{equation}
defined using the fact that in the individual (maximal) assembly maps 
$$
\mu_{\Gamma_n}:\lim_{R\to\infty}K_*^{\Ga_n}(P_R(X))\to K_*(C^*_{(max)}(\wgn)^{\Ga_n})
$$
one may arrange for the propagation of any $\mu_{\Ga_n}(x)$ to be as small as one wants (of course, all of this makes sense in more generality than our current situation).  The main ingredient in the proof of Theorem \ref{maxthe} is the following result, which is significantly more complicated than any of the ingredients used to prove Theorems 6.1 and 7.1 from the first part of this series \cite{Willett:2010ud}.

\begin{uniiso}\label{uniiso}
The uniform assembly map as in line (\ref{uniassem}) above is an isomorphism.
\end{uniiso}

The proof of this theorem is based on the Dirac-dual-Dirac method in infinite dimensions of Higson--Kasparov--Trout \cite{Higson:1999be} and Higson--Kasparov \cite{Higson:2001eb}, and its adaptation to a coarse geometric setting by the second author \cite{Yu:200ve}.  

There are several necessary preliminaries; before we embark on this, however, we show how Theorem \ref{uniiso} implies Theorem \ref{maxthe}.  The following lemma is the only point in the proof of Theorem \ref{maxthe} that necessitates the use of the \emph{maximal} Roe algebra.  It is similar to \cite[Corollary 2.11]{Oyono-Oyono:2009ua}.

\begin{unises}\label{unises}
There is a natural short exact sequence
$$
0\to \mathcal{K}(l^2(X,\Hi)) \to C^*_{max}(X)\to \frac{\prod^{U,max} C^*(\wgn)^{\Ga_n}}{\oplus C^*_{max}(\wgn)^{\Ga_n}}\to 0
$$
such that the inclusion $\mathcal{K}(l^2(X,\Hi))\to C^*_{max}(X)$ induces an injection on $K$-theory.
\end{unises}

\begin{proof}
Note first that the $C^*$-algebra $\oplus C^*_{max}(\wgn)^{\Ga_n}$ is an ideal in $\prod^{U,max} C^*(\wgn)^{\Ga_n}$, whence the right hand side of the short exact sequence makes sense.

By abuse of notation, there is a $*$-homomorphism 
$$\phi:C^*_{max}(X)\to\frac{\prod^{U,max} C^*(\wgn)^{\Ga_n}}{\oplus C^*_{max}(\wgn)^{\Ga_n}}$$
defined in just the same way as the $*$-homomorphism $\phi$ in \cite[Corollary 3.9]{Willett:2010ud}; moreover, $\mathcal{K}(l^2(X,\Hi))$ (which identifies naturally with an ideal in $C^*_{max}(X)$) is clearly in the kernel of  this map, whence $\phi$ descends to a $*$-homomorphism
$$
\phi:\frac{C^*_{max}(X)}{\mathcal{K}(l^2(X,\Hi))}\to \frac{\prod^{U,max} C^*(\wgn)^{\Ga_n}}{\oplus C^*_{max}(\wgn)^{\Ga_n}}.
$$
It thus suffices for the first part to define an inverse to this map $\phi$. 

Say then that $\mathbf{T}=(T^{(0)},T^{(1)},...)$ is an element of $\prod^U \C[\wgn]^{\Ga_n}$ such that $\text{prop}(T^{(n)})\leq R$ for all $n$.  Let $N$ be large enough so that for all $n\geq N_R$, the covering map $\pi_n:\wgn\to G_n$ is a $2R$-metric cover.  Call a pair $(x,y)\in X\times X$ \emph{$(n,R)$-good} if $n\geq N_R$, $x,y\in G_n$ and if there exist $\tilde{x},\tilde{y}\in \wgn$ such that $\pi_n(\tilde{x})=x,\pi_n(\tilde{y})=y$ and $d(\tilde{x},\tilde{y})\leq R$.
Define an element $\psi(\mathbf{T})\in \C[X]$ by the matrix coefficient formula
$$
\psi(\mathbf{T})_{x,y}:= \left\{\begin{array}{ll} T^{(n)}_{\tilde{x},\tilde{y}} & (x,y) \text{ is $(n,R)$-good} \\ 0 & \text{ otherwise}\end{array}\right.
$$
($\Ga_n$-equivariance of each $T^{(n)}$ and the $2R$-metric cover property implies that $\psi(\mathbf{T})_{x,y}$ does not depend on the choice of $\tilde{x},\tilde{y}$).  Define a map
$$
\psi:\frac{\prod^U \C[\wgn]^{\Ga_n}}{\oplus  \C[\wgn]^{\Ga_n}} \to \frac{\C[X]}{\mathcal{K}(l^2(X,\Hi))\cap \C[X]}
$$
using the formula above on operators of propagation at most $R$ for each $R>0$; it is not hard to check that this is a $*$-homomorphism.  Using the universal property of the norm on the left hand side, $\psi$ extends to a $*$-homomorphism
$$
\psi:\frac{\prod^{U,max} C^*(\wgn)^{\Ga_n}}{\oplus C^*_{max}(\wgn)^{\Ga_n}}\to \frac{C^*_{max}(X)}{\mathcal{K}(l^2(X,\Hi)},
$$
and it is not hard to check that it defines an inverse to $\phi$ on the algebraic level, whence also an inverse to $\phi$ on the $C^*$-algebraic closure.

The $K$-theoretic statement follows from \cite[Proposition 2.10]{Oyono-Oyono:2009ua}.
\end{proof}

\begin{proof}[Proof of Theorem \ref{maxthe}, assuming Theorem \ref{uniiso}]
Consider the diagram below, a close analogue of the diagram from \cite[line (8)]{Willett:2010ud}
$$
\xymatrix{ 0 \ar[d]& 0 \ar[d] \\ 
K_*(P_R(X_{N_R}))\oplus\oplus_{n\geq N_R}K_*(P_R(G_n)) \ar[d] \ar[r]& K_*(\mathcal{K}) \ar[d] \\ 
K_*(P_R(X)) \ar[d] \ar[r] & K_*(C_{max}^*(X)) \ar[d]^{\phi_*} \\
\frac{\prod K_*^{\Gamma_n}(P_R(\wgn))}{\oplus K_*^{\Gamma_n}(P_R(\wgn))} \ar[r] \ar[d] & K_*\Big(\frac{\prod^{U,max} C^*(\wgn)^{\Gamma_n}}{\oplus C_{max}^*(\wgn)^{\Gamma_n})}\Big) \ar[d] \\ 0 & 0}. 
$$
The diagram commutes, and both the left- and right hand sides are short exact sequences, the latter using Lemma \ref{unises}.  The bottom horizontal arrow is an isomorphism as $R\to\infty$ by Theorem \ref{uniiso}, the fact that the short exact sequence 
$$
0\to \oplus C_{max}^*(\wgn)^{\Gamma_n} \to \prod^{U,max} C^*(\wgn)^{\Gamma_n}\to \frac{\prod^{U,max} C^*(\wgn)^{\Gamma_n}}{\oplus C_{max}^*(\wgn)^{\Gamma_n}}\to 0
$$
gives rise to a degenerate six-term exact sequence on the level of $K$-theory, and the fact that the assembly maps 
$$
\mu_{\Gamma_n}:\lim_{R\to\infty}K_*^{\Gamma_n}(P_R(\wgn))\to K_*(C^*_{max}(\wgn)^{\Gamma_n})
$$
(which all identify with the maximal Baum-Connes assembly map for the free group $\Gamma_n$) are isomorphisms.  Finally, the top horizontal arrow is an isomorphism as $R\to\infty$, as the left hand side degenerates to being the $K$-homology of a single compact space.  The central arrow is thus an isomorphism by the five lemma, and we are done.
\end{proof}

The next three sections are devoted to the proof of Theorem \ref{uniiso}.

\section{Reformulation in terms of localization algebras}\label{locsec}

The aim of this subsection is to define a \emph{localization algebra}, an equivariant version of the machinery developed by the second author in \cite{Yu:1997kb}, and relate it to Theorem \ref{uniiso} above.  

\begin{locdef}\label{locdef}
We denote by $\prod^{U,L}\C[P_R(\wgn)]^{\Ga_n}$ the $*$-algebra of all bounded and uniformly continuous (for the norm $\|\cdot\|_{max}$) maps $\mathbf{f}$ from $[0,\infty)$ into $\prod^U\C[P_R(\wgn)]^{\Ga_n}$ such that  
$$
\text{prop}(\mathbf{f}(t))\to 0 ~~\text{ as }~~t\to\infty,
$$
where of course if we write $\mathbf{f}(t)=(f^{(0)}(t),f^{(1)}(t),...)$, then 
$$\text{prop}(\mathbf{f}(t)):=\sup_n\text{prop}(f^{(n)}(t)).$$
The \emph{localization algebra}, denoted $\prod^{U,L,max}C^*(P_R(\wgn))^{\Gamma_n}$, is the completion of $\prod^{U,L}\C[P_R(\wgn)]^{\Ga_n}$ for the norm $\|\mathbf{f}\|_{max}:=\sup_{t\in[0,\infty)}\|\mathbf{f}(t)\|_{max}$.
\end{locdef}

Note that the localization algebra depends on \emph{both} the local and large-scale structure of the spaces $P_R(\wgn)$; cf. \cite[Remark 3.3]{Willett:2010ud}.

Now, there is an assembly map
\begin{equation}\label{locassem}
\mu_L:\lim_{R\to\infty}\prod K_*(P_R(\wgn))\to \lim_{R\to\infty}K_*(\prod^{U,L,max}C^*(P_R(\wgn))^{\Gamma_n});
\end{equation}
defined by taking an operator $F=\prod F^{(n)}$ representing a cycle on the left hand side, building a sequence of operators $(F_m)_{m\in\N}$ from it with propagation tending to zero,  interpolating between them, and then taking a $K$-theoretic index (as in \cite[Definition 4.2]{Willett:2010ud}) of the resulting operator.

The proof of the following proposition uses Lipschitz homotopy invariance and a Mayer-Vietoris argument; as it is essentially the same as the argument in the non-equivariant case from \cite{Yu:1997kb}, the proof is omitted.

\begin{lociso}\label{lociso}
The local assembly map as in line (\ref{locassem}) above is an isomorphism. \qed
\end{lociso}

Note that $\prod^{U,L,max}C^*(P_R(\wgn))^{\Gamma_n}$ is equipped with an `evaluation-at-zero' $*$-homomorphism
\begin{equation}\label{eval}
e:\prod^{U,L,max}C^*(P_R(\wgn))^{\Gamma_n}\to \prod^{U,max}C^*(P_R(\wgn))^{\Gamma_n}
\end{equation}
defined in the obvious way.  Moreover, these evaluation maps pass to the direct limit as $R$ tends to infinity, and fit into a diagram
$$
\xymatrix{\lim_{R\to\infty}\prod K_*^{\Ga_n}(P_R(\wgn)) \ar[d]^{\mu_U} \ar[r]^(.42){\mu_L} & \lim_{R\to\infty}K_*(\prod^{U,L,max}C^*(P_R(\wgn))^{\Gamma_n}) \ar[d]^{e_*} \\  K_*(\prod^{U,max}C^*(\wgn)^{\Gamma_n}) \ar[r]^(.42){\cong}  & \lim_{R\to\infty} K_*(\prod^{U,max}C^*(P_R(\wgn))^{\Gamma_n})},
$$
here we have used the existence of non-canonical isomorphisms
$$
\prod^{U,max}C^*(P_R(\wgn))^{\Gamma_n}\cong \prod^{U,max}C^*(\wgn)^{\Gamma_n}
$$
which induce canonical isomorphisms on $K$-theory -- the idea here is the same as that behind \cite[Lemma 3.7]{Willett:2010ud} -- to produce the isomorphism in the bottom row; this diagram clearly commutes by definition of all the maps involved. The following corollary, which we state as a theorem, is thus immediate.

\begin{locthe}\label{locthe}
The evaluation at zero map
$$
e:\lim_{R\to\infty}\prod^{U,L,max}C^*(P_R(\wgn))^{\Gamma_n}\to \lim_{R\to\infty}\prod^{U,max}C^*(P_R(\wgn))^{\Gamma_n}
$$
as in line (\ref{eval}) induces an isomorphism on $K$-theory if and only if the uniform assembly map
$$
\mu_U:\lim_{R\to\infty}\prod_nK_*^{\Ga_n}(P_R(\wgn))\to K_*(\prod^{U,max}_n C^*(\wgn)^{\Ga_n})
$$
from line (\ref{uniassem}) above is an isomorphism. \qed
\end{locthe}

In the next two sections, we will prove that the evaluation-at-zero map $e_*$ is an isomorphism.

\section{Isomorphism for twisted algebras}\label{twistsec}

In this subsection, we define twisted versions of the uniform products from the previous section
$$\umtwist$$
as well as twisted versions of the localization algebras
$$\ulmtwist;$$
just as in the previous section, there is then a twisted version of the evaluation-at-zero map  
\begin{equation}\label{twiste}
e:\ulmtwist\to \umtwist.
\end{equation}
Here each $\mathcal{V}_n$ is a real Hilbert space equipped with a proper isometric $\Gamma_n$ action and an equivariant coarse embedding $f_n:\wgn\to\mathcal{V}_n$ (cf.\ \cite[Definition 3.4]{Willett:2010ud}); $\mathcal{V}_n$ is built directly from the tree $\wgn$ using a well-known construction of Julg and Valette \cite{Julg:1984rr}.  $\ahn$ is then the $C^*$-algebra of a Hilbert space, defined by Higson--Kasparov--Trout \cite{Higson:1999be}, and should be thought of as providing `proper coefficients' for the Roe algebras and localization algebras.

The $C^*$-algebras $\ahn$ and twisted Roe algebras built from them are naturally graded.  It will be convenient in this section and the next for $K_*(A)$ to denote the \emph{graded} $K$-theory groups of a graded $C^*$-algebra $A$; of course, if $A$ has the trivial grading, then its graded $K$-theory is the same as its usual $K$-theory.

The following theorem is an analogue of the fact that the Baum-Connes conjecture with proper coefficients always holds (see \cite[Chapter 13]{Guentner:2000fj} and \cite{Chabert:2001ye} for the latter).

\begin{mvthe}\label{mvthe}
The map induced on $K$-theory by the direct limit as $R\to\infty$ of the twisted evaluation-at-zero maps from line \eqref{twiste} above,
$$e_*:\lim_{R\to\infty}K_*(\ulmtwist)\to \lim_{R\to\infty}K_*(\umtwist),$$
is an isomorphism.
\end{mvthe}
\noindent 
The proof uses a Mayer-Vietoris argument similar to that used by the second author in \cite[Section 6]{Yu:200ve}, but made equivariant and kept uniform over all $n$.  

We now begin with the preliminaries.  We start in a fairly general setting, partly as the objects are of interest in their right, and partly to keep the notation under control.  The following definition introduces the $C^*$-algebra of a Hilbert space.

\begin{c*h}\label{c*h}
Let $\V$ be a real (countably infinite dimensional) Hilbert space.
Denote by $V_a,V_b$ etc.\ the finite dimensional affine subspaces of $\V$.  Let $V_a^0$ be the linear subspace of $V$ consisting of differences of elements of $V_a$.  Let $\Cliff_\C(V_a^0)$ be the complexified Clifford algebra of $V_a^0$ and $\mathcal{C}(V_a)$ the graded $C^*$-algebra of continuous functions vanishing at infinity from $V_a$ into $\Cliff_\C(V_a^0)$.  Let $\mathcal{S}$ be the $C^*$-algebra $C_0(\R)$, graded by taking the even and odd parts to consist of even and odd functions respectively, and define $\mathcal{A}(V_a):=\mathcal{S}\hat{\otimes}\mathcal{C}(V_a)$ (throughout, `$\hat{\otimes}$' denotes the graded spatial tensor product of graded $C^*$-algebras, or the completed graded tensor product of graded Hilbert spaces as appropriate).   

If $V_a\subseteq V_b$, denote by $V_{ba}^0$ the orthogonal complement of $V_{a}^0$ in $V_b^0$.  One then has a decomposition $V_b=V_{ba}^0+V_a$ and corresponding (unique) decomposition of any $v_b\in V_b$ as $v_b=v_{ba}+v_{a}$.  Any function $h\in\mathcal{C}(V_a)$ can thus be extended to a multiplier $\tilde{h}$ of $\mathcal{C}(V_b)$ by the formula 
$$\tilde{h}(v_{ba}+v_a)=h(v_a)\in\Cliff_\C(V_a^0)\subseteq\Cliff_\C(V_b^0).$$

Continuing to assume that $V_a\subseteq V_b$, denote by $C_{ba}:V_b\to\Cliff_\C(V_{ba}^0)$ the function $v_b\mapsto v_{ba}$ where $v_{ba}$ is considered as an element of $\Cliff_\C(V_{ba}^0)$ via the inclusion $V_{ba}^0\subseteq \Cliff_\C(V_{b}^0)$.  Let $X$ be the unbounded multiplier of $\mathcal{S}$ given by the function $t\mapsto t$.  Define a $*$-homomorphism $\beta_{ba}:\mathcal{A}(V_a)\to\mathcal{A}(V_b)$ via the formula
\begin{displaymath}
\beta_{ba}(g\hat{\otimes}h)=g(X\hat{\otimes}1+1\hat{\otimes}C_{ba})(1\hat{\otimes}\tilde{h}),
\end{displaymath}  
where $g\in\mathcal{S}$, $h\in\mathcal{C}(V_a)$ and the term on the right involving $g$ is defined using the functional calculus for unbounded multipliers.

These maps make the collection $(\mathcal{A}(V_a))$ as $V_a$ ranges over finite dimensional affine subspaces of $V$ into a directed system.  Define the \emph{$C^*$-algebra of $\V$} to be
\begin{displaymath}
\mathcal{A}(\V)=\lim_\rightarrow\mathcal{A}(V_a)
\end{displaymath} 
\end{c*h}

Now, for any finite dimensional $V_a\subseteq \V$, $C_0(V_a\times\R_+)$ is included in $\mathcal{A}(V_a)$ as its center.  It follows that  the center $Z(\mathcal{A}(\mathcal{V}))$ is $C_0(\V\times\R_+)$, where $\mathcal{V}\times\R_+$ is equipped with the weakest topology such that the projection to $\mathcal{V}$ is weakly continuous, and so that the functions
$$
(w,t)\mapsto t^2+\|v-w\|^2
$$
as $v$ ranges over $\mathcal{V}$ are continuous.  This makes $\mathcal{V}\times\R_+$ into a locally compact Hausdorff space in which the `balls' 
\begin{equation}\label{ball}
B_r(v):=\{(w,t)\in\V\times\R_+~|~t^2+\|v-w\|^2<r\}
\end{equation}
are open and the `closed balls'
$$\overline{B}_r(v):=\{(w,t)\in\V\times\R_+~|~t^2+\|v-w\|^2\leq r\}$$
are compact; both of these statements follow from the fact that the function $(w,t)\mapsto e^{-(t^2+\|v-w\|^2)}$ is a $C_0$-function on $\V\times\R_+$ for the topology above -- indeed it is the image of the function $e^{-t^2}\in\mathcal{S}\cong \mathcal{A}(\{v\})$ under the $*$-homomorphism $\mathcal{A}(\{v\})\to\mathcal{A}(\V)$ defined by the construction of the latter algebra as a direct limit.  For the remainder of this section and the next $\V\times\R_+$ is always considered with this topology.

\begin{suppv}\label{suppv}
The \emph{support} of an element $a\in\mathcal{A}(\V)$ is the complement of all points $(v,t)\in\V\times\R_+$ such that there exists $g\in C_0(\V\times\R_+)$ with $g(v,t)\neq0$ and $g\cdot a=0$.  

If $O$ is an open subset of $\V\times\R_+$, define $\mathcal{A}(O)$ to be the closure in $\mathcal{A}(\V)$ of the $*$-subalgebra of $\mathcal{A}(\V)$ consisting of all elements with support in $O$. $\mathcal{A}(O)$ is then a closed ideal in $\mathcal{A}(\V)$.
\end{suppv}

Now, assume that $\V$ is a real Hilbert space equipped with an affine isometric action of a countable discrete group $\Ga$.  This action gives rise to an action of $\Ga$ on $\mathcal{A}(\V)$ by $*$-automorphisms, and a compatible action on $C_0(\V\times\R_+)$ by homeomorphisms (of course this latter is just given by the formula $g\cdot (v,t)=(g\cdot v,t)$, with the obvious notation). Let $Y$ be a uniformly discrete bounded geometry metric space equipped with an isometric action of $\Ga$ and an equivariant coarse embedding $f:Y\to \V$.  For each $R>0$, $f$ may be extended to a continuous map $f:P_R(Y)\to\V$ by stipulating that $f$ preserve convex combinations.  Assuming $Y$ is `coarsely geodesic' in any reasonable sense (we will eventually only be interested in trees, which certainly have this property), this extension $f$ is still an equivariant coarse embedding.

\begin{dubyas}\label{dubyas}
Let $x$ be a point in $P_R(Y)$ for some $R>0$.  For each $k\in\N$ define
$$
W_k(x):=f(x)+\text{span}\{f(y)-f(x)~|~y\in P_R(Y) \text{ and } d_{P_R(Y)}(x,y)\leq k^2\},
$$
a finite dimensional affine subspace of $\V$ (this uses that $f$ is a coarse embedding, and bounded geometry of $Y$).  

Denote by $\beta_k(x):\mathcal{A}(W_k(x))\to \mathcal{A}(\V)$ the $*$-homomorphism coming from the definition of $\mathcal{A}(V)$ as a directed system as in Definition \ref{c*h} above, and write $\beta(x)$ for 
$$\beta_0(x):\mathcal{S}\cong\mathcal{A}(W_0(x))\to\mathcal{A}(\V).$$
\end{dubyas}

The following definition gives the twisted Roe algebras that form the basis for the argument in this section.

\begin{twisted}\label{twisted}
Let $f,Y,\Ga,\V$ be as above.  For each $R>0$, choose a countable dense $\Ga$-equivariant subset of $P_R(Y)$, say $Z=Z_R$, just as we did when defining $\C[X]^{\Ga}$ in  \cite[Definition 3.6]{Willett:2010ud}.  Assume moreover that $Z_R\subseteq Z_{R'}$ whenever $R\leq R'$.
Define $\C[P_R(Y);\mathcal{A}(\V)]^{\Ga}$ to be the collection of $Z\times Z$ indexed matrices $(T_{x,y})$ such that each $T_{x,y}$ is an element of $\mathcal{A}(\V)\hat{\otimes}\mathcal{K}$ and such that:
\begin{enumerate}
\item for all $(x,y)\in Z\times Z$ and all $g\in\Ga_n$, $T_{gx,gy}=g\cdot T_{x,y}$;
\item there exists $L>0$ so that for each $x\in Z$ the cardinalities of the sets
$$\{z\in Z~|~T_{x,z}\neq 0\} ~~\text{ and }~~\{z\in Z~|~T_{z,x}\neq0\}$$
are both at most $L$;
\item there exists $M\geq 0$ so that $\|T_{x,y}\|\leq M$ for all $x,y\in Z$;
\item there exists $r_1>0$ so that $T_{x,y}=0$ whenever $d(x,y)> r_1$;
\item there exists $r_2>0$ so that for all $x,y\in Z\times Z$, $\text{supp}(T_{x,y})\subseteq B_{r_2}(f(x))$;
\item there exist $k,K>0$ such that for each $x,y\in Z$ there exists $T'_{x,y}\in\mathcal{A}(W_k(x))\hat{\otimes}\mathcal{K}$ such that $T_{x,y}=(\beta_k(x)\hat{\otimes}1)(T'_{x,y})$ and moreover so that $T'_{x,y}$ is a finite linear combination of at most $K$ elementary tensors from
$$\mathcal{A}(W_k(x))\hat{\otimes}\mathcal{K}\cong \mathcal{S}\hat{\otimes}C_0(W_k(x),\Cliff_\C(W_k(x)^0))\hat{\otimes}\mathcal{K};$$
\item there exists $c>0$ such that if $T'_{x,y}$ is as above, and $w\in W_k(x)\times\R_+$ is of norm one then the derivative of $T'(x,y)$ in the direction of $w$, $\nabla_{w} T'_{x,y}$, exists in $\mathcal{A}(W_k(x))\hat{\otimes}\mathcal{K}$ and is of norm at most $c$.
\end{enumerate}
$\C[P_R(Y);\mathcal{A}(\V)]^{\Ga}$ is then made into a $*$-algebra using the usual matrix operations, and $*$-algebra operations on $\mathcal{A}(\V)\hat{\otimes}\mathcal{K}$.
\end{twisted}

\begin{supp}\label{supp}
The \emph{support} of an element $T\in \C[P_R(Y);\mathcal{A}(\V)]^{\Ga}$ is the set
$$\text{supp}(T):=\{(x,y,v,t)\in Z\times Z\times \V\times\R_+~|~(v,t)\in \text{supp}(T_{x,y})\}.$$
If $O$ is a $\Ga$-invariant open subset of $\V$, $\C[P_R(Y);\mathcal{A}(\V)]^{\Ga}_O$ is defined to be the $*$-ideal of $\C[P_R(Y);\mathcal{A}(\V)]^{\Ga}$ consisting of elements with support in $Z\times Z\times O$.
\end{supp}

We now specialize back to the situation of interest: a sequence of trees $(\wgn)_{n\in\N}$ equipped with free, proper, isometric and cocompact actions of (necessarily free) groups $\Ga_n$.  The following construction, due to Julg and Valette \cite{Julg:1984rr}, is fundamental; it could equivalently be performed using negative type functions, but we prefer this more direct approach.

Let $T$ be the vertex set of a tree, and $E$ its oriented edge set.  For an edge $e\in E$, denote by $-e$ the same edge, but with the opposite orientation.  Define $\Omega(T)$ to be the quotient of the real Hilbert space of square summable functions on $E$ by the closed subspace spanned by elements of the form $\delta_e+\delta_{-e}$ (one thinks of $\Omega(T)$ as $l^2$-sections of the `tangent bundle' of $T$).  Thus in $\Omega(T)$, $\delta_{-e}=-\delta_e$.  Fix a basepoint $b\in T$, and for any $x\in T$, let $\text{geod}(x)\subseteq E$ be the collection of edges on the (unique) edge geodesic from $b$ to $x$ (oriented to point from $b$ to $x$).  Define now a map $f:T\to \Omega(T)$ by 
$$
f:x\mapsto \sum_{e\in \text{geod}(x)}\delta_e,
$$
and note the equality 
$$
\|f(x)-f(y)\|_{\Omega(T)}^2=d_T(x,y)
$$
for all $x,y\in T$ (a corollary of the fact that the triangle $(b,x,y)$ in the tree $T$ looks like a tripod); in particular, $f$ is a coarse embedding.  Moreover, if a group $\Gamma$ acts properly on $T$ by isometries, then the formula
$$
\alpha_g:\delta_e\mapsto \delta_{g\cdot e}+f(g\cdot b)
$$ 
defines a proper affine isometric action on $\Omega(T)$, for which the embedding $f:T\to \Omega(T)$ is equivariant (this again used the `tripod' quality of triangles in a tree).

Applying this to the trees $\wgn$ and groups $\Ga_n$ gives a sequence of real Hilbert spaces $\mathcal{V}_n$ equipped with proper affine isometric actions of $\Ga_n$ and equivariant coarse embeddings $f_n:\wgn\to \mathcal{V}_n$ such that 
\begin{equation}\label{distort}
\|f_n(x)-f_n(y)\|_{\mathcal{V}_n}^2=d_{\wgn}(x,y)
\end{equation}
and so that the affine subspace generated by $f_n(\wgn)$ is dense in $\V_n$ (and actually a linear subspace, as $f_n(b)=0$).
  For any $R>0$, we extend $f_n$ to $P_R(\wgn)$ by stipulating that $f_n$ preserve convex combinations; the resulting maps $f_n:P_R(\wgn)\to \mathcal{V}_n$ remain equivariant, and while they no longer satisfy the precise equalities in line (\ref{distort}) above, they are still coarse embeddings with respect to uniform constants across the entire sequence $(\wgn)_{n\in\N}$ (this uses bounded geometry of the original sequence $(G_n)$).

\begin{unitwist}\label{unitwist}
Define $\uatwist$ to be the set of sequences $\mathbf{T}=(T^{(0)},T^{(1)},...)$ such that each $T^{(n)}$ is an element of $\C[P_R(\wgn);\ahn]^{\Ga_n}$ and so that conditions 2--7 in Definition \ref{twisted} are satisfied by each $T^{(n)}$ with respect to the same constants.  This set is given a $*$-algebra structure using pointwise operations.

The \emph{twisted maximal Roe algebra}, denoted
$$\umtwist,$$
is defined to be the completion of $\uatwist$ for the norm
$$\|\mathbf{T}\|_{max}=\sup\{\|\pi(T)\|_{\mathcal{B}(\mathcal{H})}~|~\pi:\uatwist\to \mathcal{B}(\mathcal{H}) \text{ a $*$-representation}\}$$
in the usual way; it is not hard to use uniform bounded geometry of the trees $\wgn$ to check that the norm is finite. 

Let moreover $\ualtwist$ denote the $*$-algebra of all bounded, uniformly continuous (for the norm $\|\cdot\|_{max}$ defined above) maps
$$
\mathbf{f}:[0,\infty)\to \uatwist
$$
such that there are uniform constants with respect to which the sequences of operators $\mathbf{f}(t)$ satisfy conditions 2--7 in Definition \ref{twisted} for all $t$, and so that if
$$\mathbf{f}(t)=(f^{(0)}(t),f^{(1)}(t),...)\in \uatwist$$
then $\sup_{n}\text{prop}(f^{(n)}(t))\to0$ as $t\to\infty$.  Define 
$$
\ulmtwist
$$
to be the completion of $\ualtwist$ for the norm $\|\mathbf{f}\|_{max}:=\sup_t\|\mathbf{f}(t)\|_{max}$.
\end{unitwist}

The proof of Theorem \ref{mvthe} proceeds by a Mayer-Vietoris argument, which is used to reduce the problem to a study of particularly simple pieces.  The following definition restricts the algebras from Definition \ref{unitwist} down to a subset.

\begin{restwist}\label{restwist}
Let $O=(O_n)_{n\in\N}$ be a sequence of sets such that each $O_n$ is a $\Ga$-invariant open sets of $\V_n\times\R_+$.  We define 
$$
\uatwist_O
$$
to be the collection of sequences $(T^{(0)},T^{(1)},...)$ in $\uatwist$ such that each $T^{(n)}$ is in $\C[P_R(\wgn);\ahn]^{\Ga_n}_{O_n}$ as in Definition \ref{suppv}.  The $C^*$-algebra
$$
\umtwist_O
$$
is then defined to be the closure of $\uatwist_O$ in $\umtwist$.  

Similarly, 
$$
\ualtwist_O
$$
is defined to be the collection of $\mathbf{f}\in\ualtwist$ such that $f^{(n)}(t)\in \C[P_R(\wgn);\ahn]^{\Ga_n}_{O_n}$ for all $t,n$, and 
$$
\ulmtwist_O
$$
is its closure in $\ulmtwist$.
\end{restwist} 

The pieces we actually use are as in the following definition.

\begin{basicp}\label{basicp}
Fix $R>0$.  Let $r>0$ and $k\in\N$.  An open subset $O_n\subseteq \V_n$ is called \emph{($r,k$)-basic} if it can be written as a finite disjoint union of orbits
$$
O_n=\sqcup_{i=1}^k \Ga_n\cdot U_{n,i}
$$
where for each $i$ there exists $x_{n,i}\in P_R(\wgn)$ such that $U_{n,i}\subseteq B_r(f(x_{n,i}))$, and so that $U_{n,i}$ is the pullback to $\V_n\times\R_+$ of an open ball in $W_k(x_{n,i})\times\R_+$ under the natural map
$$
\V_n\times\R_+\to W_k(x_{n,i})\times\R_+
$$
coming from the definition of $\V_n\times\R_+$ as a projective limit (in particular, $U_{n,i}$ is open).

A collection $O=(O_n)_{n\in\N}$, where each $O_n$ is an open subset of $\mathcal{V}_n$ is called \emph{basic} if there exist $r,k$ such that each $O_n$ is ($r,k$)-basic.
\end{basicp}

\begin{basiclem}\label{basiclem}
Let $O$ be a basic collection.  Then the restricted evaluation-at-zero map
$$
e_*: \lim_{R\to\infty}K_*(\ulmtwist_O) \to \lim_{R\to\infty}K_*(\umtwist_O)
$$
is an isomorphism.
\end{basiclem}

The proof is similar to \cite[Lemma 6.4]{Yu:200ve}; in order to keep this piece relatively self-contained, we will give a complete proof.  The proof is in any case in some respects simpler than \cite[Lemma 6.4]{Yu:200ve}, due to the fact that `equivariance forces uniformity'.  The proof requires some preliminaries.

Fix for the moment a basic collection $O=(O_n)$, and write 
$$
O_n=\sqcup_{i=1}^{k_n}\Ga_n\cdot U_{n,i}
$$
where each $U_{n,i}$ is an open set as in Definition \ref{basicp} contained in some $B_r(f(x_{i,n}))\subseteq \V_n\times\R_+$.  Fix for the moment $R>0$ and let 
\begin{equation}\label{ycol}
Y=(\{Y_{n,i}\}_{i=1}^{k_n})_{n\in\N}
\end{equation}
be such that:
\begin{itemize}
\item each $Y_{n,i}$ is a closed subset of $P_R(\wgn)$;
\item there exists $s$ such that $\text{diameter}(Y_{n,i})\leq s$ for all $n,i$;
\item $x_{n,i}$ is an element of $Y_{n,i}$ for each $n,i$.
\end{itemize}
Note that we do not assume that the collection of subsets $Y_{n,1},...,Y_{n,k_n}$ of $P_R(\wgn)$ is disjoint.

For such a collection $Y$ define $A(Y)$ to be the $*$-algebra of sequences $\mathbf{T}=(T^{(0)},T^{(1)},...)$ such that for each $n$
$$
T^{(n)}=(T^{(n,1)},...,T^{(n,k_n)})\in\oplus_{i=1}^{k_n} C^*(Y_{n,i})\hat{\otimes}\mathcal{A}(U_{n,i})
$$
and so that conditions 2--7 from Definition \ref{twisted} are satisfied by all of the operators $T^{(n,i)}$ uniformly.  Note that conditions 4 and 5 are redundant, however, by uniform boundedness of the sets $Y_{n,i}$.  Let $A^*(Y)$ be the completion of $A(Y)$ for the norm 
$$
\|\mathbf{T}\|:=\sup_{n,i}\|T^{(n,i)}\|_{C^*(Y_{n,i})\hat{\otimes}\mathcal{A}(O_{n,i})}.
$$

Similarly, define $A_L(Y)$ to be the $*$-algebra of bounded, uniformly continuous (with respect to the norm above) maps $\mathbf{f}:[0,\infty)\to A(Y)$ such that if we write
$$\mathbf{f}(t)=((f^{(n,i)}(t))_{i=1}^k)_{n\in\N}$$
then 
$$
p(t):=\sup_{n,i}(\text{prop}(f^{(n,i)}(t)))
$$
exists and tends to zero as $t$ tends to infinity.  Define $A^*_L(Y)$ to be the completion of $A_L(Y)$ for the norm $\|\mathbf{f}\|=\sup_t\|\mathbf{f}(t)\|$.  Note that there is of course an evaluation-at-zero map
\begin{equation}\label{aezero}
e:A_L^*(Y)\to A^*(Y).
\end{equation}

The next three definitions are essentially from \cite{Yu:1997kb}.

\begin{lipdef}\label{lipdef}
Let $Y=\{Y_{n,i}\}$ and $Y'=\{Y'_{n,i}\}$ be two collections of subsets as in line \eqref{ycol} above.  A collection of maps $g=\{g_{n,i}:Y_{n,i}\to Y_{n,i}'\}$ is said to be a \emph{Lipschitz map from $Y$ to $Y'$} if there exists some $c\geq 0$ such that each $g_{n,i}$ is $c$-Lipschitz.  Composition of Lipschitz maps is defined component-wise in the obvious way.
\end{lipdef}

\begin{lipisom}\label{lipisom}
Let $g=\{g_{n,i}:Y_{n,i}\to Y_{n,i}'\}$ be a Lipschitz map from $Y$ to $Y'$ as in the previous definition.
Let $Z_{n,i}\subseteq Y_{n,i}$ be the countable dense subset used to define $C^*(Y_{n,i})$ and similarly for $Z'_{n,i}\subseteq Y'_{n,i}$.  Note in particular that for each $n,i$, $C^*(Y_{n,i})$ is represented on $\mathcal{H}_{n,i}:=l^2(Z_{n,i},\Hi)$ for some fixed separable infinite dimensional Hilbert space $\Hi$, and similarly for $\mathcal{H}'_{n,i}$.

Let now $(\epsilon_m)_{m\in\N}$ be any sequence of positive real numbers that converges to zero. For each $n,i$ and positive integer $m$, there exists an isometry
$$V_{n,i,m}:\mathcal{H}_{n,i}\to\mathcal{H}_{n,i}'$$
such that if $\text{supp}(V_{n,i,m})$ is the complement of the set of all $(y,y')\in Y_{n,i}\times Y_{n,i}'$ such that there exist $h\in C(Y_{n,i})$ and $h'\in C(Y_{n,i}')$ with $h(y)\neq 0\neq h'(y')$ and $hV_{n,i,m}h'=0$, then 
$$
\text{supp}(V_{n,i,m})\subseteq \{(y,y')\in Y_{n,i}\times Y_{n,i}'~|~d(g(y),y')<\epsilon_m\}.
$$
This construction uses a standard partition of unity argument (see for example \cite[Section 4, Lemma 2]{Higson:1993th}).

Define now 
$$V_m=\oplus_{n\in\N}\oplus_{i=1}^{k_n}V_{n,i,m}:\oplus_{n,i}\mathcal{H}_{n,i}\to\oplus_{n,i}\mathcal{H}_{n,i}.$$
For $t\in[0,1]$, let  
$$
R(t)=\begin{pmatrix} \cos(\frac{\pi}{2}t) & \sin(\frac{\pi}{2}t) \\ -\sin(\frac{\pi}{2}t) & \cos(\frac{\pi}{2}t) \end{pmatrix}
$$
and define a $t$-parametrized family of isometries
$$
V_g(t):(\oplus_{n,i}\mathcal{H}_{n,i})\oplus(\oplus_{n,i}\mathcal{H}_{n,i})\to (\oplus_{n,i}\mathcal{H}'_{n,i})\oplus(\oplus_{n,i}\mathcal{H}'_{n,i})
$$
via the formula 
$$
V_g(t)=R(t-m+1)\begin{pmatrix} V_m & 0 \\ 0 & V_{m+1} \end{pmatrix}R(t-m+1)^*
$$
whenever $m-1\leq t<m$.  Any family $V_g(t)$ constructed in this way is called a \emph{covering isometry for $g$}.

Let now $A_{L,0}^*(Y)$ be the kernel of the evaluation $*$-homomorphism $e$ in line \eqref{aezero} above, and $A_{L,0}^*(Y)^+$ its unitization.  Define finally a $*$-homomorphism 
$$Ad(V_g):A^*_{L,0}(Y)^+\to M_2(A^*_{L,0}(Y')^+)$$
by the formula
$$
Ad(V_g):a+\lambda I \mapsto V_g(t)\begin{pmatrix} a(t) & 0 \\ 0 & 0 \end{pmatrix}V_g(t)^*+\lambda I.
$$

It is not hard to check that the map on $K$-theory induced by $Ad(V_g)$ does not depend on any of the choices involved in its construction (essentially the same argument as in \cite[Section 4, Lemma 3]{Higson:1993th} applies).   We denote it by
$$g_*=Ad(V_g)_*:K_*(A^*_{L,0}(Y))\to K_*(A^*_{L,0}(Y)).$$  
Note that this association is `functorial' in the sense that if $g,g'$ are two Lipschitz maps as above, then
$$
g_*\circ g'_*=Ad(V_g)_*\circ Ad(V_{g'})_*=Ad(V_{g}V_{g'})_*=Ad(V_{g\circ g'})_*=(g\circ g')_*.
$$  
\end{lipisom}

\begin{liphom}\label{liphom}
Say that $\{g^0_{n,i}\}$ are $\{g^1_{n,i}\}$ are two Lipschitz maps from $Y$ to $Y'$.  They are said to be \emph{strongly Lipschitz homotopic} if there exists a collection of maps $\{F_{n,i}:Y_{n,i}\times[0,1]\to Y'_{n,i}\}$ such that
\begin{itemize}
\item $F_{n,i}(\cdot, j)=g^j_{n,i}$ for all $n,i$ and $j=0,1$;
\item there exists $c>0$ such that each $F_{n,i}$ is $c$-Lipschitz when restricted to slices of the form $Y_{n,i}\times\{t\}\subseteq Y_{n,i}\times[0,1]$;
\item the collection of restrictions
$$\{F_{n,i}|_{\{x\}\times[0,1]}~|~n\in\N,i=1,...,k_n,x\in Y_{n,i}\}$$
is equicontinuous.
\end{itemize}

Two collections $Y$, $Y'$ are said to be \emph{strong Lipschitz homotopy equivalent} if there exist Lipschitz maps $g:Y\to Y'$ and $g':Y'\to Y$ such that the Lipschitz maps $g'\circ g$ and  $g\circ g'$ are strongly Lipschitz homotopic to the identity maps on $Y$ and $Y'$ respectively.
\end{liphom}

The next two lemmas are similar to \cite[Lemmas 6.5 and 6.6]{Yu:200ve}.  

\begin{liplem}\label{liplem}
Let $Y$ be as in line \eqref{ycol} above, and $A_{L,0}^*(Y)$ be the kernel of the evaluation-at-zero map as in line \eqref{aezero} above.  If $Y$ and $Y'$ are strong Lipschitz homotopy equivalent, then the $K$-theory groups $K_*(A^*_{L,0}(Y))$ and $K_*(A^*_{L,0}(Y'))$ are isomorphic.
\end{liplem}

\begin{proof}
Using Definition \ref{lipisom}, it certainly suffices to show that if $g=\{g_{n,i}:Y_{n,i}\to Y_{n,i}\}$ is a Lipschitz map that is strongly Lipschitz homotopic to the identity, then the map induced on $K$-theory $g_*:K_*(A^*_{L,0}(Y))\to K_*(A^*_{L,0}(Y))$ is the identity (functoriality and symmetry complete the argument).  Let then $F=\{F_{n,i}:Y_{n,i}\times[0,1]\to Y_{n,i}\}$ be such that $F_{n,i}(y,0)=g_{n,i}(y)$ and $F_{n,i}(y,1)=y$ for all $n,i$ and $y\in Y_{n,i}$.

For any $k,l\in\N$, define
$$
s_{k,l}=\left\{\begin{array}{ll} \frac{k}{l+1} & k\leq l+1 \\ 1 & k>l+1 \end{array}\right.,
$$
and note there exists a sequence $(\epsilon_l)_{l\in\N}$ of positive numbers that tends to zero, and so that for all ${n,i}$ and all $y\in Y_{n,i}$,
\begin{equation}\label{fk}
d_{Y_{n,i}}(F_{n,i}(s_{k+1,l},y),F_{n,i}(s_{k,l},y))<\epsilon_l
\end{equation}
and 
\begin{equation}\label{fl}
d_{Y_{n,i}}(F_{n,i}(s_{k,l+1},y),F_{n,i}(s_{k,l},y))<\epsilon_l.
\end{equation}
Let now $\mathcal{H}_{n,i}$ be the Hilbert space used to define $C^*(Y_{n,i})$ in the usual way and for each $k,l,n,i$ let 
$$V_{n,i,k,l}:\mathcal{H}_{n,i}\to\mathcal{H}_{n,i}$$
be any isometry such that 
$$
\text{supp}(V_{n,i,k,l})\subseteq \{(x,y)\in Y_{n,i}\times Y_{n,i}~|~d(F(s_{k,l},x),y)<\epsilon_l\}
$$ 
for all $k,l$ and so that $V_{n,i,k,l}=I$ for $k\geq l+1$.  Define further 
$$
V_{k,l}=\oplus_{n,i}V_{n,i,k,l}:\oplus_{n,i}\mathcal{H}_{n,i}\to\oplus_{n,i}\mathcal{H}_{n,i}
$$
and finally a family of isometries
$$
V_k(t):(\oplus_{n,i}\mathcal{H}_{n,i})\oplus(\oplus_{n,i}\mathcal{H}_{n,i})\to (\oplus_{n,i}\mathcal{H}_{n,i})\oplus(\oplus_{n,i}\mathcal{H}_{n,i})
$$
(where $t$ is now taken in $[0,\infty)$) by the formula
$$
V_k(t)=R(t-l+1)\begin{pmatrix} V_{k,l-1} & 0 \\ 0 & V_{k,l}\end{pmatrix}R(t-l+1)^*,
$$
for $l-1\leq t<l$, where $R(t)$ is as in Definition \ref{lipisom}.  Note that each $V_k$ is a covering isometry for $g$ as in that definition, and in particular defines a $*$-homomorphism
$$
Ad(V_k):A^*_{L,0}(Y)^+\to M_2( A^*_{L,0}(Y)^+)
$$
as described there (this uses the property in line \eqref{fl}) above.

Now, $A^*_{L,0}(Y)$ is stable, whence any element of $K_1(A^*_{L,0}(Y))$ can be represented by a single unitary $u\in A^*_{L,0}(Y)^+$.  Note that we may use Hilbert spaces
$$\mathcal{H}_{n,i}^\infty:=\oplus_{k=0}^\infty\mathcal{H}_{n,i}$$
in place of the $\mathcal{H}_{n,i}$ to define a new $C^*$-algebra $A^*_{L,0}(Y_0)_\infty$, which is abstractly isomorphic to $A^*_{L,0}(Y_0)$, and into which the latter algebra embeds naturally as a corner; this embedding induces an isomorphism on $K$-theory.  From now on we work inside $A^*_{L,0}(Y_0)_\infty$, in particular considering $u$ as an element of $A^*_{L,0}(Y_0)_\infty^+$ via this corner embedding.  

Form
$$
a:=\oplus_{k\geq0}Ad(V_k)(u)\begin{pmatrix} u^* & 0 \\ 0 & I \end{pmatrix};
$$
that $a$ is a unitary element in $M_2(A^*_{L,0}(Y_0)_\infty)$ follows from the fact that for any fixed $l$, 
$$(Ad(V_k)(u))(l)=\begin{pmatrix} u(l) & 0 \\ 0 & 0 \end{pmatrix}$$ 
for all $k$ suitably large.  Using property \eqref{fk} above, $a$ is equivalent in $K$-theory to
$$
b:=\oplus_{k\geq 1}Ad(V_k)(u)\begin{pmatrix} u^* & 0 \\ 0 & I \end{pmatrix}
$$
which is in turn clearly equivalent to 
$$
c:=I \oplus\oplus_{k\geq 1}Ad(V_k)(u)\begin{pmatrix} u^* & 0 \\ 0 & I \end{pmatrix}.
$$
Hence in $K_1(A^*_{L,0}(Y_0)_\infty)$ we have the identities
$$
0=[a]-[c]=[ac^*]=\Big[Ad(V_0)(u)\begin{pmatrix} u^* & 0 \\ 0 & I\end{pmatrix}\oplus_{k\geq 1} I\Big]=g_*[u]-[u],
$$
using the fact that $V_0(t)$ is a covering isometry for $g$.  Hence $g_*[u]=[u]$ in $K_1(A^*_{L,0}(Y_0)_\infty)$, whence also in $K_1(A^*_{L,0}(Y_0))$.  The case of $K_0$ can be handled similarly using a suspension argument, and we are done.
\end{proof}

\begin{simlem}\label{simlem}
Let $Y=\{Y_{n,i}\}$ be a collection as in line \eqref{ycol} above, where each $Y_{n,i}$ is a single simplex.  Then the evaluation-at-zero map as in line \eqref{aezero} above induces an isomorphism on $K$-theory.
\end{simlem}

\begin{proof}
Considering the short exact sequence 
$$
0\to A^*_{L,0}(Y)\to A^*_L(Y)\stackrel{e}{\to}A^*(Y)\to0,
$$
it is enough to show that $K_*(A^*_{L,0}(Y))=0$.  Let $Y_0=((\{x_{n,i}\})_{i=1}^{k_n})_{n\in\N}$ be the collection where each element is the singleton $\{x_{n,i}\}$; as the collections $Y$ and $Y_0$ are clearly strong Lipschitz homotopy equivalent, it suffices by Lemma \ref{liplem} to prove that $K_*(A^*_{L,0}(Y_0))=0$.  This we will do using an Eilenberg swindle.  We will consider only $K_1(A^*_{L,0}(Y_0))$; the case of $K_0$ can be handled similarly using a suspension argument.

Now, just as in the proof of Lemma \ref{liplem}, stability of $A^*_{L,0}(Y_0)$ implies that any element of $K_1(A^*_{L,0}(Y_0))$ can be represented by a single unitary, say $u$, in the unitization $A^*_{L,0}(Y_0)^+$.  For each $s\in[0,\infty)$, consider
$$
u_s(t):=\left\{\begin{array}{ll} I & 0\leq t\leq s \\ u(t-s) & s\leq t \end{array}\right.
$$
(note that as $u\in A^*_{L,0}(Y_0)^+$, the two halves match up continuously at $t=s$).  For each $n,i$, let $\mathcal{H}_{n,i}$ denote the Hilbert space used in the definition of $C^*(\{x_{n,i}\})$ (thus $\mathcal{H}_{n,i}$ is an infinite dimensional separable Hilbert space, with the unit action of $C(\{x_{n,i}\})\cong\C$).  Just as in the proof of Lemma \ref{liplem}, use new Hilbert spaces
$$\mathcal{H}_{n,i}^\infty:=\oplus_{k=0}^\infty\mathcal{H}_{n,i}$$
to define a new $C^*$-algebra $A^*_{L,0}(Y_0)_\infty$, which is abstractly isomorphic to $A^*_{L,0}(Y_0)$, and into which the latter algebra embeds naturally as a corner; this embedding induces an isomorphism on $K$-theory.

Define now 
$$
u_\infty=\oplus_{k=0}^\infty u_k\in A^*_{L,0}(Y_0)_\infty^+
$$
(we use here that $Y_0$ is a union of single points to note that the propagation of each $u_\infty$ is controlled -- in fact, zero).  On the level of $K$-theory, however,
$$
[u]+[u_\infty]=[u \oplus \oplus_{k=1}^\infty u_k]=[u_\infty]
$$
by homotoping all the $u_k$s `one step to the right'.  Hence $[u]=0$ in $K_1(A^*_{L,0}(Y_0)_\infty)$, and so also in $K_1(A^*_{L,0}(Y_0))$.
\end{proof}

\begin{isolem}\label{isolem}
Let 
$$
O=(\sqcup_{i=1}^{k_n}\Ga\cdot U_{n,i})_{n\in\N}
$$
be our original basic collection.  For each $R>0$, $n\in\N$ and $i=1,...,k_n$, let $\Delta_{n,i}^R$ be the simplex in $P_R(\wgn)$ with vertices $\{x\in\wgn ~|~d(x,x_{n,i})\leq R/2\}$, and $Y_\Delta^R$ be the collection $\{\Delta^R_{n,i}\}$ (which of course satisfies the conditions after line \eqref{ycol}).  Then 
$$
\lim_{R\to\infty}\umtwist_O\cong \lim_{R\to\infty}A^*(Y_\Delta^R)
$$ 
and
$$
\lim_{R\to\infty}\ulmtwist_O\cong \lim_{R\to\infty}A_L^*(Y_\Delta^R);
$$ 
moreover, these isomorphisms commute with the natural evaluation at zero maps.
\end{isolem}

\begin{proof}
Fix for the moment $R>0$, and for each $S>0$, let $Y^S$ be the collection of sets as in line \eqref{ycol} above defined by $Y^S_{n,i}=\{x\in P_R(\wgn)~|~d(x,x_{n,i})\leq S\}$.  Define now $A(Y^S)^\Ga$ to be the $*$-algebra of all sequences $(T^{(0)},T^{(1)},...,)$
such that each $T^{(n)}$ is of the form
$$
T^{(n)}=(g\cdot T^{(n,1)},...,g\cdot T^{(n,k_n)})_{g\in\Ga}\in \prod_{g\in\Ga}C^*(g \cdot Y^S_{n,i})\hat{\otimes} \mathcal{A}(g\cdot U_i),
$$
where the $T^{(n,i)}$ are just as in the definition of $A(Y^S)$.  Clearly, $A(Y^S)$ is $*$-isomorphic to $A(Y^S)^\Ga$, and moreover it is not hard to check that there is an algebraic isomorphism
$$
\uatwist\cong \lim_{S\to\infty}A(Y^S)^\Ga,
$$
whence
$$
\uatwist\cong \lim_{S\to\infty}A(Y^S).
$$
However, the unitization of the algebra on the right is inverse-closed inside the unitization of $\lim_{S\to\infty}A^*(Y^S)$, whence its universal closure simply is the $C^*$-algebra $\lim_{S\to\infty}A^*(Y^S)$.  The same is thus true for $\uatwist$ of course, i.e. the above map extends to a $*$-isomorphism
$$
\umtwist\cong \lim_{S\to\infty}A^*(Y^S)
$$
of $C^*$-algebras.  The first displayed line in the lemma follows directly from this, and the second is similar.
\end{proof}
 
Finally, we are ready to give the proof of Lemma \ref{basiclem}.

\begin{proof}[Proof of Lemma \ref{basiclem}]
From Lemma \ref{isolem} it is enough to prove that the evaluation-at-zero maps induce isomorphisms
$$
e_*:K_*(A^*_L(Y_\Delta^R))\to K_*(A^*(Y^R_\Delta))
$$
for all $R$.  This is immediate from Lemma \ref{simlem}, however.
\end{proof}

Now, for any $r>0$ and $n\in\N$, define
\begin{equation}\label{gnr}
(\wgn)_{r}=\cup_{x\in\wgn}B_r(f_n(x))
\end{equation}
to be the `generalized $r$-neighborhood' of $\wgn$ in $\mathcal{V}_n\times\R_+$.  The following lemma splits up each $(\wgn)_{r}$ into basic pieces.

\begin{splitlem}\label{splitlem}
For each $r>0$ there exists $k_r>0$ and $N_r\in\N$ such that for each $n$ one can write 
$$(\wgn)_{r}=\cup_{j=1}^{N_r}O_r^j,$$
where each $O_r^j$ is an  ($r,k_r$)-basic set as in Definition \ref{basicp}.
\end{splitlem}

\begin{proof}
Note that each $\Ga_n$ acts on $(\wgn)_r$ cocompactly; using the fact that pullbacks of balls as in the definition of a basic set form a basis for the topology on $\V_n\times\R_+$, it is not hard to see from here that for each $n$ there exists some $N_n$ so that $(\wgn)_r$ is covered by a union of $N_n$ $(r,k_n)$-basic sets for some $k_n>0$.  However, for all $n\geq M$ for some $M$ suitably large, the covering faithfulness property of the sequence $(\wgn)_{n\in\N}$ together with the uniformity of the coarse embeddings $f_n:\wgn\to \V_n$ guarantees that for any $x\in \wgn$, $\Ga_n\cdot B_r(f_n(x))$ is a disjoint union of the different sets $g\cdot B_r(f_n(x))$ as $g$ ranges over $\Ga_n$.  From this and uniform bounded geometry of the spaces $\wgn$, it follows that there exists $N_M$ such that for each $n\geq M$, $\wgn$ can be written as a union of at most $N_M$ $(r,0)$-basic sets, each of which is of the form
$$
O=\sqcup_{i=1}^{j}\Ga_n\cdot B_r(f_n(x_i))
$$
for some finite set $\{x_1,...,x_j\}\subseteq \wgn$.  Take $N_r=\max\{N_0,N_1,...,N_M\}$ and $k_n=\min\{k_0,...,k_{M-1}\}$.
\end{proof}

The next lemma relates algebraic operations on the algebras $\umtwist_O$ to set-theoretic operations on the collections $O$.  It is a close analogue of \cite[Lemma 6.3]{Yu:200ve}.

\begin{insum}\label{insum}
Fix $r>0$, $k\in\N$ and $N\in\N$.  For each $j=1,...,N$, let $(O^j_n)_{n\in\N}$ be a basic collection with respect to the parameters $r,k$; write
$$O^j_n=\sqcup_i \Ga\cdot U_{n,i}^j,$$
where the $U_{n,i}^j$ are as in the definition of an ($r,k$)-basic set; in particular, each $U_{n,i}^j$ may be considered as a subset of $W_k(x_{n,i}^j)\times\R_+$ for some $x^j_{n,i}\in P_R(\wgn)$.

For each $s>0$ and $n,i,j$, let $_sU_{n,i}^j$ denote the interior $s$-neighbourhood of $U_{n,i}^j$, i.e.\
$$
_sU_{n,i}^j:=\{x\in W_k(x_{n,i}^j)\times\R_+~|~d(x,(W_k(x_{n,i}^j)\times\R_+)\backslash U_{n,i}^j)>s\},
$$
an open (possibly empty) subset of $W_k(x_{n,i}^j)\times\R_+$, which we will also think of as an open subset of $\V_n\times\R_+$ via pullback.

Define a new (basic) collection $_sO^j$ to have n$^\text{th}$ component
$$
_sO^j_n:=\sqcup_i \Ga\cdot~ _sU_{n,i}^j.
$$
Let $_sO=(_sO_n)_{n\in\N}$ be the collection (not necessarily a basic collection) given by 
$$
{}_sO_n=\cup_{j=1}^{N-1}{}_sO^j_n.
$$
Then 
\begin{align*}
\lim_{s\to0}& \umtwist_{_sO\cup _sO^N} \\ &=\Big(\lim_{s\to0}\umtwist_{_sO}\Big)+\Big(\lim_{s\to0}\umtwist_{_sO^N}\Big)
\end{align*}
and 
\begin{align*}
\lim_{s\to0}&\umtwist_{_sO\cap _sO^N} \\ &=\Big(\lim_{s\to0}\umtwist_{_sO}\Big)\cap\Big(\lim_{s\to0}\umtwist_{_sO^N}\Big)
\end{align*}
and similarly in the case of the localization algebras.
\end{insum}

\begin{proof}
The intersection case is not difficult, so we focus on the sum case.  Moreover, the case of the localization algebras is similar, so we will only actually prove the case in the first line of the conclusion above.  Further, the inclusion
\begin{align*}
\lim_{s\to0}& \umtwist_{_sO\cup _sO^N} \\ &\supseteq\Big(\lim_{s\to0}\umtwist_{_sO}\Big)+\Big(\lim_{s\to0}\umtwist_{_sO^N}\Big)
\end{align*}
is clear, so it suffices to prove the converse inclusion, and moreover it suffices to prove this on the algebraic level, i.e.\ to show that
\begin{align*}
\lim_{s\to0}& \uatwist_{_sO\cup _sO^N} \\ &\subseteq\Big(\lim_{s\to0}\uatwist_{_sO}\Big)+\Big(\lim_{s\to0}\uatwist_{_sO^N}\Big);
\end{align*}
this is what we will actually prove.

Let then $\mathbf{T}=(T^{(0)},T^{(1)},...)$ be an element of $\uatwist_{_sO\cup _sO^N}$ for some fixed $s$.  For each $n,i,j$, let $c_{n,i}^j$ be such that
\begin{enumerate}
\item $c^j_{n,i}$ is a smooth map from $W_k(x^j_{n,i})\times\R_+$ to $[0,1]$;
\item $c^j_{n,i}$ is identically one on $_{\frac{s}{2}}U^j_{n,i}$;
\item $\text{supp}(c^j_{n,i})\subseteq _{\frac{s}{3}}U^j_{n,i}$.
\end{enumerate}
Note in particular that conditions 1 and 3 imply that each $c^j_{n,i}$ can be considered as an element of the center of $\mathcal{A}(W_k(x_{n,i}^j))$.
Define 
$$
g^j_{n,i}=\beta_k(x^j_{n,i})(c^j_{n,i})\in\mathcal{A}(\V_n\times\R_+),
$$
and for each $x\in P_R(\wgn)$, define 
$$
g_{n,x}=\sum_{j=1}^{N-1}\sum_{g\in\Ga_n}g\cdot\Big(\sum_{d(x,x^j_{n,i})\leq \text{prop}(T^{(n)})} g^j_{n,i}\Big)
$$
and 
$$
g_{n,x}^N=\sum_{g\in\Ga_n}g\cdot\Big(\sum_{d(x,x^j_{n,i})\leq \text{prop}(T^{(n)})} g^N_{n,i}\Big)
$$
(these make sense, are uniformly bounded, and have uniformly bounded derivatives using uniform bounded geometry of the $\wgn$). Define $h_{n,x}$ and $h_{n,x}^N$ precisely analogously, but starting with functions $d^j_{n,i}$ that satisfy the same conditions as the $c^j_{n,i}$ but with $s/2$ and $s/3$ in conditions 2 and 3 replaced by $s/3$ and $s/4$ respectively.  

Define finally sequences of operators $\mathbf{A}=(A^{(0)},A^{(1)},...)$ and $\mathbf{B}=(B^{(0)},B^{(1)},...)$ by
$$
A^{(n)}_{x,y}=\frac{g_x}{h_x+h_x^N}T^{(n)}_{x,y}~~ \text{ and } ~~B^{(n)}_{x,y}=\frac{g_x^N}{h_x+h^N_x}T^{(n)}_{x,y},
$$
and note that $\mathbf{T}=\mathbf{A}+\mathbf{B}$.  It is not difficult to use the comments above to check that 
$$
\mathbf{A}\in \uatwist_{_{(s/3)}O}~~ \text{ and } ~~\mathbf{B}\in \uatwist_{_{(s/3)}O^N}
$$
so we are done.
\end{proof}

There is a well-known Mayer-Vietoris sequence in $K$-theory associated to a pushout square: cf.\ for example \cite[Section 3]{Higson:1993th}.  Applying this to the previous lemma gives the following Mayer-Vietoris sequence.

\begin{mvcor}\label{mvcor}
With the set up as in the previous lemma, denote by
\begin{align*}
A & =\lim_{s\to0}\umtwist_{_sO\cup _sO^N} \\
B & =\lim_{s\to0}\umtwist_{_sO} \\
C & =\lim_{s\to0}\umtwist_{_sO^N} \\
D & =\lim_{s\to0}\umtwist_{_sO\cap _sO^N}.
\end{align*}
Then there exists a (six-term cyclic) Mayer-Vietoris sequence
$$
\cdots\to K_i(D)\to K_i(B)\oplus K_i(C)\to K_i(A)\to K_{i-1}(D)\to \cdots 
$$
and similarly for the localization algebras. \qed
\end{mvcor}

One can now use a Mayer-Vietoris argument as in \cite[Theorem 6.8]{Yu:200ve} to complete the proof of Theorem \ref{maxthe}.

\begin{proof}[Proof of Theorem \ref{mvthe}]
For each $n\in\N$ and $r>0$, let $(\wgn)_r$ be as in line \eqref{gnr} above.  As one has that 
$$
\umtwist=\lim_{r\to\infty}\umtwist_{(\wgn)_r}
$$
and 
$$
\ulmtwist=\lim_{r\to\infty}\ulmtwist_{(\wgn)_r},
$$
it suffices to prove that the restricted evaluation map 
$$
e_*:K_*(\ulmtwist_{(\wgn)_r})\to K_*(\umtwist_{(\wgn)_r})
$$
is an isomorphism for all $r>0$, and indeed that the restricted evaluation map
$$
e_*:K_*(\lim_{s\to0}\ulmtwist_{(\wgn)_{r-s}})\to K_*( \lim_{s\to0}\umtwist_{(\wgn)_{r-s}})
$$
This, however, follows from Lemma \ref{basiclem}, Lemma \ref{splitlem}, Corollary \ref{mvcor} and induction, so we are done.
\end{proof}

\section{The Dirac-dual-Dirac method in infinite dimensions}\label{dddsec}

In this subsection we construct a commutative diagram
\begin{equation}\label{ddd}
\xymatrix{ \lim_{R}K_*(\prod^{U,L,max}C^*(P_R(\wgn))^{\Gamma_n}) \ar[r]^{e_*} \ar[d]^{\beta_*} & \lim_RK_*(\prod^{U,max}C^*(P_R(\wgn))^{\Gamma_n}) \ar[d]^{\beta_*} \\
\lim_RK_*(\ulmtwist) \ar[r]^{e_*} \ar[d]^{\alpha_*} & \lim_RK_*(\umtwist) \ar[d]^{\alpha_*} \\
\lim_RK_*(\umlk) \ar[d]^{\gamma_*^{-1}} \ar[r]^{e_*} & \lim_RK_*(\umk) \ar[d]^{\gamma_*^{-1}} \\
\lim_{R}K_*(\prod^{U,L,max}C^*(P_R(\wgn))^{\Gamma_n}) \ar[r]^{e_*} & \lim_RK_*(\prod^{U,max}C^*(P_R(\wgn))^{\Gamma_n}) 
}
\end{equation}
where the maps labeled $\alpha_*$ and $\beta_*$ are analogues of the \emph{Dirac} and \emph{Bott} (or \emph{dual Dirac} elements) used by Higson--Kasparov in \cite{Higson:2001eb}.   The maps labeled $\gamma_*^{-1}$ are such that $\gamma_*$ is the composition of $\alpha_*$ and $\beta_*$, and should be thought of as being the identity on $K$-theory (although their domains and ranges are in fact slightly different).

The main theorem of this section, implicit in the preceding discussion is as follows.

\begin{dddthe}\label{dddthe}
The vertical compositions in diagram (\ref{ddd}) above are isomorphisms (even before taking the limits as $R\to\infty$).
\end{dddthe}

Theorem \ref{uniiso}, whence Theorem \ref{maxthe}, follows immediately from this, Theorem \ref{mvthe}, Theorem \ref{locthe}, and a diagram chase.  

The proof of Theorem \ref{dddthe} is based on \cite[Section 7]{Yu:200ve}, using properness of the $\Ga_n$-actions to ensure that everything works equivariantly, and the uniformity of the coarse embeddings $f_n:\wgn\to\V_n$ to ensure that the Dirac and Bott morphisms considered there only alter the propagations of operators by a uniform amount over the entire sequence $(\wgn)_{n\in\N}$.

The maps in diagram \ref{ddd} above will all be constructed as \emph{asymptotic morphisms} \cite{Guentner:2000fj}.  The reader is referred to the memoir \cite{Guentner:2000fj} and \cite{Higson:1999be,Yu:200ve,Higson:2001eb,Higson:2004la} for background, and the sources of most of the ideas behind the current section.

\subsection*{The Dirac map $\alpha$}

We begin with the definition of $\alpha$.  Working in generality for the moment, let $\V$ be a separable infinite dimensional Hilbert space.  Let $V\subseteq\V$ denote a finite dimensional affine subspace of $\V$ and $V^0$ the corresponding finite dimensional linear subspace of differences of elements from $V$.  Let $\mathcal{L}^2(V):=L^2(V,\Cliff_\C(V^0))$ denote the graded Hilbert space of $L^2$-maps from $V$ to the complex Clifford algebra of $V^0$, $\Cliff_\C(V^0)$  (here we use the inner product on $\V$ to define a `Lebesgue measure' on $V$ and an inner product on $\Cliff_\C(V^0)$ in order to make sense of this).   

Say $V_a\subseteq V_b$ are finite dimensional affine subspaces of $\V$.  Define
$$
V_{ba}^0=V^0_b\ominus V_a^0
$$
to be the orthogonal complement of $V_a$ in $V_b$, which is a linear subspace of $\V$. Define $\xi_0\in \mathcal{L}^2(V_{ba}^0)$ by
\begin{equation}\label{xi0}
\xi_0(w)=\pi^{-\text{dim}(V_{ba})/4}\exp\Big(-\frac{1}{2}\|w\|^2\Big),
\end{equation}
and an isometric inclusion
\begin{align*}
v_{ba}:\mathcal{L}^2(V_a) \to \mathcal{L}^2(V_{ba}^0) \hat{\otimes}\mathcal{L}^2(V_a)\cong \mathcal{L}^2(V_b)
\end{align*}
by 
\begin{equation}\label{isoinc}
v_{ba}:\xi \mapsto \xi_0\hat{\otimes}\xi;
\end{equation}
it is not difficult to check that these isometries are compatible in the sense that $v_{cb}\circ v_{ba}=v_{ca}$ whenever the composition makes sense, whence they turn the collection 
$$
\{\mathcal{L}^2(V)~|~V\subseteq \V\text{ a finite dimensional affine subspace}\}
$$
into a directed system.  Define 
\begin{equation}\label{hil}
\mathcal{L}^2(\V):=\lim_\to \mathcal{L}^2(V),
\end{equation}
where the limit is taken over the directed system of all affine subspaces of $\V$ as above.  

Let $\sw(V)\subseteq \mathcal{L}^2(V)$ be the (dense) subspace of Schwartz class functions from $V$ to $\Cliff_\C(V^0)$.  Choose an orthonormal basis $\{v_1,...,v_n\}$ for $V^0$; using a fixed choice of basepoint $v\in V$, these define coordinates $x_i$ on $V$ by the `duality relationships'
$$
x_i: v+v_j\mapsto \delta_{ij},
$$
where $\delta_{ij}$ is the Kronecker delta.
Define the \emph{Dirac operator}, an unbounded differential operator on $\mathcal{L}^2(V)$ with domain $\sw(V)$, by the formula
$$
D_{V}=\sum_{i=1}^n \frac{\partial}{\partial x_i} v_i,
$$
where $v_i$ is thought of as acting by Clifford multiplication; $D_{V}$ does not depend on the choice of orthonormal basis or on the basepoint $v$.  Define moreover the \emph{Clifford operator}, also an unbounded operator on $\mathcal{L}^2(V)$ with domain $\sw(V)$, by the formula
$$
(C_{V,v}\xi)(w)=(w-v)\cdot \xi(w),
$$ 
where $v\in V$ is again a fixed basepoint, and where the multiplication is to be thought of as Clifford multiplication by the vector $w-v\in V^0$; note, of course that $C_{V,v}$ does depend on the choice of basepoint.

We now specialize back to the case of interest.  Fix $n\in\N$ and denote by $\mathcal{L}^2_n$ the Hilbert space $\mathcal{L}^2(\V_n)$ constructed in line \eqref{hil} above.   Note that $\Ga_n$ has a unitary action on $\mathcal{L}^2_n$ coming from the affine isometric action of $\Ga_n$ on $\V_n$.    Fix now $x\in P_R(\wgn)$, let $W_k(x)$ be as in Definition \ref{dubyas} above, and denote by 
$$
v_{k}:\mathcal{L}^2(W_k(x))\to \mathcal{L}^2(W_{k+1}(x))
$$
the isometric inclusion from line \eqref{isoinc} above.  Note that these inclusions preserve the Schwartz subspaces $\sw(W_k(x))$, and define a Schwartz subspace of $\mathcal{L}^2_n$ by taking the algebraic direct limit
\begin{equation}\label{bigs}
\sw(x)=\lim_{k\to\infty}\sw(W_k(x));
\end{equation}
this vector space depends on the choice of $x$, but will always be a dense subspace of $\mathcal{L}^2_n$.  Denote by $\mathcal{K}(\mathcal{L}^2_n)$ the graded $C^*$-algebra of compact operators on $\mathcal{L}^2_n$, and define new $C^*$-algebras
$$
\umk~~\text{ and } ~~\umlk
$$
precisely analogously to the $C^*$-algebras
$$
\prod^{U,max}C^*(P_R(\wgn)))^{\Ga_n} ~~\text{ and } \prod^{U,L,max}C^*(P_R(\wgn)))^{\Ga_n}
$$
from Definitions \ref{uniprod} and \ref{locdef} above, only now with $\mathcal{K}(\Hi)\hat{\otimes}\mathcal{K}(\mathcal{L}^2_n)$ used instead of $\mathcal{K}(\Hi)$ for matrix entries; of course, the only real difference is the non-triviality of the grading (this is not so important), and the non-triviality of the $\Ga_n$ action on the former algebras of compact operators.  

For each $k\in\N$ such that $k\geq 1$, define 
$$
V_k(x):=W_{k+1}(x)^0\ominus W_{k}(x)^0
$$
to be the orthogonal complement of $W_k(x)$ in $W_{k+1}(x)$ and define $V_0(x)=W_1(x)$.  We may then consider the Dirac operators, denoted
$$D_{k}:=D_{V_0(x)},$$
and Clifford operators, denoted
$$C_{k,x}:=\left\{\begin{array}{ll} C_{V_0(x),f(x)} & k=0 \\ C_{V_k(x),0} & k\geq 1\end{array}\right.,$$ 
associated to each $V_k(x)$ in the manner above.  Define for each $m\in\N$ and each $t\geq 1$ an operator $B_{m,t}(x)$ by 
$$
B_{m,t}(x)=\sum_{k=0}^{m-1} (1+kt^{-1})D_{k}+\sum_{k=m}^\infty (1+kt^{-1})(D_{k}+C_{k,x});
$$
note that as the function $\xi_0\in \mathcal{L}^2(V_k(x))$ from line \eqref{xi0} above is in the kernel of $D_{k}+C_{k,x}$ for all $k$, the operator $B_{m,t}(x)$ is well-defined on the Schwartz space $\sw(x)$ as in line \eqref{bigs} above; we take this for its domain.  Note that the collection of operators $\{B_{m,t}(x)\}_{x\in P_R(\wgn)}$ is `equivariant' in the sense that any $g\in\Ga_n$ maps the domain $\sw(x)$ of $B_{m,t}(x)$ to the domain $\sw(gx)$ of $B_{m,t}(gx)$, and conjugates the former operator to the latter operator.  

Now, for each $x\in P_R(\wgn)$ and each $k\in\N$, the algebra $\mathcal{C}(W_k(x))$ from Definition \ref{c*h} is represented on $\mathcal{L}^2_n$ via its natural representation on $\mathcal{L}^2(W_k(x))$ and the isometric inclusion of this Hilbert space into $\mathcal{L}^2_n$.  If $h$ is an element of $\mathcal{C}(W_k(x))$ and $V\supseteq W_k(x)$ is a finite dimensional affine subspace of $\V_n$, define an operator $\widetilde{h}$ acting on $\mathcal{L}^2(V)$ by 
$$
(\widetilde{h}\cdot \xi)(w+v)=h(w)\cdot \xi(w+v),
$$
where $\xi$ is an element of $\mathcal{L}^2(V)$, and $w+v$ is an element of $V$ written in such a way that $w\in W_k(x)$ and $v$ is in the orthogonal complement of $W_k(x)^0$ in $V^0$ (of course, a unique decomposition of this form exists for any element of $V$). Define moreover for each $h\in \mathcal{C}(W_k(x))$, $g\in\mathcal{S}$ and $t\in [1,\infty)$ elements $h_t$ and $g_t$ via the formulas
$$
g_t(s)=g(t^{-1}s)~~\text{ and } h_t(v)=h(f(x)+t^{-1}(v-f(x)))
$$
respectively.  Define for each $m\in\N$ and $t\in[1,\infty)$ a map $\theta^m_t(x)$ from the collection of finite linear combinations of elementary tensors in 
$$\mathcal{A}(W_m(x))\hat{\otimes}\mathcal{K}(\Hi)\cong \mathcal{S}\hat{\otimes}\mathcal{C}(W_m(x))\hat{\otimes}\mathcal{K}(\Hi)$$
to $\mathcal{K}(\Hi)\hat{\otimes}\mathcal{K}(\mathcal{L}^2_n)$
by the formula
\begin{equation}\label{theta}
\theta_t^m:(g\hat{\otimes}h)\hat{\otimes}k\mapsto g_t(B_{m,t}(x)|_{W_m(x)})\widetilde{h_t}\hat{\otimes}k
\end{equation}
on elementary tensors of $g\in\mathcal{S}$, $h\in \mathcal{C}(W_m(x))$, and $k\in\mathcal{K}(\Hi)$, and extending by linearity.

Define moreover for each $t\in[1,\infty)$ a map 
$$\alpha_t:\uatwist\to \umk$$ 
by the formula
$$
(\alpha_t(\mathbf{T}))^{(n)}_{x,y}:=\theta^m_t(x)(T'_{x,y})
$$
for each $n\in\N$ and $x,y\in Z_R\subseteq P_R(\wgn)$, where $m$ and $T'_{x,y}\in\mathcal{A}(W_m(x))\hat{\otimes}\mathcal{K}(\Hi)$ are such that  
$$
\beta_m(x)(T'_{x,y})=T^{(n)}_{x,y}
$$
as in part 6 of Definition \ref{twisted} (we will show in the proof of Lemma \ref{alphalem} below that the choice of $m$ does not matter).  Application of $\alpha_t$ as above pointwise, i.e.\ using the formula
$$(\alpha_t\mathbf{f})(s):=\alpha_t(\mathbf{f}(s))$$
for $s\in [0,\infty)$ similarly defines maps
$$
\alpha_t : \ualtwist\to \umlk.
$$  

The proof of the following lemma is analogous to that of \cite[Lemma 7.2]{Yu:200ve}.

\begin{alphalem}\label{alphalem}
The maps $\alpha_t$ above extend to asymptotic morphisms on the $C^*$-algebraic completions
$$
\alpha_t:\umtwist\to \umk
$$
and 
$$
\alpha_t:\ulmtwist\to \umlk. \eqno 
$$
\end{alphalem}

\begin{proof}
We will only consider the case of 
$$
\alpha_t:\umtwist\to \umk;
$$
the case of the localization algebras is similar.  Note that it follows from the remarks we have already made on equivariance and the local compactness of the operators $g(B_{m,t}(x)|_{W_m(x)})$ (which in turn follows from ellipticity and the Rellich lemma -- see for example \cite[page 8]{Higson:1999be}) that the image of $\alpha_t$ really is in $\umk$.

Fix $m\in\N$ and $R>0$, and $K,r,c>0$.  Let $\epsilon>0$.  Let $x$ be any element of any $P_R(\wgn)$.  Denote by 
$$
\mathcal{A}(W_m(x))\hat{\otimes}\mathcal{K}(\Hi)_{K,r,c}
$$
the subset of $\mathcal{A}(W_m(x))\hat{\otimes}\mathcal{K}(\Hi)$ consisting of all elements of the form
$$
\sum_{i=1}^Kg_i\hat{\otimes}h_i\hat{\otimes}k_i
$$
where $g_i\in\mathcal{S}$, $h_i\in \mathcal{C}(W_m(x))$, $k_i\in\mathcal{K}(\Hi)$ and such that
\begin{itemize}
\item each $g_i$ is supported in $[-r,r]$;
\item each $g_i$ and $h_i$ are continuously differentiable and satisfy $\|g_i'\|\leq c$, $\|\nabla_wh_i\|\leq c$ for all $w\in W_k(x)$ such that $\|w-f(x)\|\leq 1$
\end{itemize}
(in other words, the set of elements satisfying most of the conditions on matrix entries from Definition \ref{twisted} uniformly).  It follows from \cite[Lemma 7.5]{Yu:200ve} (which in turn uses \cite[Lemma 2.9]{Higson:1999be}) and the uniformly bounded geometry of the sequence $(\wgn)_{n\in\N}$ that there exists $t_0>0$ so that for all $t>t_0$ and all $a,b\in \mathcal{A}(W_m(x))\hat{\otimes}\mathcal{K}(\Hi)_{K,r,c}$,
$$
\|\theta^m_t(ab)-\theta^m_t(a)\theta^m_t(b)\|,~~\|\theta^m_t(a^*)-\theta^m_t(a)^*\|<\epsilon.
$$

Now, consider $\mathcal{A}(W_m(x))\hat{\otimes}\mathcal{K}(\Hi)$ and $\mathcal{A}(W_m(y))\hat{\otimes}\mathcal{K}(\Hi)$ as subalgebras of $\ahn\hat{\otimes}\mathcal{K}(\Hi)$, and use this to make sense of their intersection.  From the above, it follows that to show that 
$$
\alpha_t : \uatwist\to \umk
$$
is an asymptotic morphism, equivalently, defines a $*$-homomorphism into the asymptotic algebra 
\begin{equation}\label{aa}
\mathfrak{A}(\umk):=\frac{C_b([1,\infty),\umk)}{C_0([1,\infty),\umk)}
\end{equation}
(see \cite[Definition 1.1]{Guentner:2000fj}), it suffices to show that for any $r_1>0$ and $m,K,r,c>0$ as above and $\epsilon>0$ there exists $t_0>0$ such that if $x,y\in P_R(\wgn)$ satisfy $d(x,y)\leq r_1$ then for any
$$a\in \mathcal{A}(W_m(x))\hat{\otimes}\mathcal{K}(\Hi)_{K,r,c}\cap \mathcal{A}(W_m(y))\hat{\otimes}\mathcal{K}(\Hi)_{K,r,c}$$
and all $t>t_0$,
$$
\|\theta^m_t(x)(a)-\theta^m_t(y)(a)\|<\epsilon
$$
(this uses the uniform bounded geometry property of $(G_n)_{n\in\N}$ again, and from here the fact that there are only uniformly finitely many non-zero entries in any row or column of the sort of finite propagation matrices that we are dealing with).
This, however, follows from \cite[Lemmas 7.3 and 7.4, and proof of Lemma 7.2]{Yu:200ve}.

Finally, we must show that $\alpha_t$ extends from a $*$-homomorphism
$$
\alpha_t : \uatwist\to \mathfrak{A}(\umk)
$$
to the $C^*$-algebraic closure $\umtwist$ of the left hand side; this is immediate from the universal property of the maximal norm, however, so we are done.
\end{proof}

We denote by 
$$
\alpha_*: K_*(\umtwist)\to K_*(\umk)
$$
and 
$$
\alpha_*:K_*(\ulmtwist)\to K_*(\umlk)
$$
the maps induced by these asymptotic morphisms on $K$-theory.

\subsection*{The Bott map $\beta$}

For each $t\in[1,\infty)$, define a map
$$
\beta_t :\mathcal{S}\hat{\odot}\prod^U\C[\wgn]^{\Ga_n}\to \umtwist
$$
by the formula
$$
(\beta_t (g\hat{\otimes}\mathbf{T}))^{(n)}_{x,y}=\beta(x)(g_t)\hat{\otimes} T^{(n)}_{x,y},
$$
where 
$$\beta(x):\mathcal{S}\cong\mathcal{A}(\{f(x)\})\to \mathcal{A}(\V_n)$$
is the $*$-homomorphism from Definition \ref{dubyas} above.  Applying the above maps pointwise similarly defines a family of maps
$$
\beta_t:\mathcal{S}\hat{\odot}\prod^{U,L}\C[\wgn]^{\Ga_n}\to \ulmtwist.
$$

The proof of the following lemma is similar to that of \cite[Lemma 7.6]{Yu:200ve}.

\begin{betalem}\label{betalem}
The maps $\beta_t$ defined above extend to asymptotic morphisms
$$
\beta_t:\mathcal{S}\hat{\otimes}\prod^{U,max}C^*(\wgn)^{\Ga_n}\to \umtwist
$$
and 
$$
\beta_t:\mathcal{S}\hat{\otimes}\prod^{U,L,max}C^*(\wgn)^{\Ga_n}\to \ulmtwist.
$$
\end{betalem}

\begin{proof}
Again, we only consider the case of 
$$
\beta_t:\mathcal{S}\hat{\otimes}\prod^{U,max}C^*(\wgn)^{\Ga_n}\to \umtwist;
$$
the case of localization algebras is similar.  Note first that each of the individual $\beta^{(n)}_t$s has image in the $\Ga_n$-invariant part of 
$$C^*_{max}(\wgn;\ahn)\subseteq \umtwist,$$
whence the image of each $\beta_t$ is indeed in $\umtwist$.

The fact that $\beta_t$ defines a $*$-homomorphism into the asymptotic algebra 
$$
\mathfrak{A}(\umtwist):=\frac{C_b([1,\infty),\umtwist)}{C_0([1,\infty),\umtwist)}
$$
(cf.\ line \eqref{aa} above) follows from the argument of \cite[Lemma 3.2]{Higson:2004la} combined with that of \cite[Lemma 7.3]{Yu:200ve} to show that for all $R,r,c,\epsilon>0$ there exists $t_0$ such that for all $n$ and all $x,y\in P_R(\wgn)$ such that $d(x,y)\leq r$, all $t>t_0$ and all $g\in \mathcal{S}$ such that $\text{supp(g)}\subseteq [-r,r]$ and $\|g'(s)\|\leq c$ then
$$
\|\beta(x)(g_t)-\beta(y)(g_t)\|<\epsilon.
$$
The fact that $\beta_t$ extends to a $*$-homomorphism
$$
\beta_t:\mathcal{S}\hat{\otimes}\prod^{U,max}C^*(\wgn)^{\Ga_n}\to \mathfrak{A}(\umtwist)
$$ 
now follows from the universal property of the norm on $\prod^{U,max}C^*(\wgn)^{\Ga_n}$, the universal property of the maximal tensor product, and nuclearity of $\mathcal{S}$ so that the maximal tensor product agrees with the spatial tensor product.
\end{proof}

We denote by 
$$
\beta_*:K_*(\prod^{U,max}C^*(\wgn)^{\Ga_n})\to K_*(\umtwist)
$$
and 
$$
\beta_*:K_*(\prod^{U,L,max}C^*(\wgn)^{\Ga_n})\to K_*( \ulmtwist)
$$
the corresponding homomorphisms induced on $K$-theory.

\subsection*{The Gamma map $\gamma$}

We now define the last of our asymptotic morphisms $\gamma$.  It is a very close analogue of the usual `$\gamma$-element' appearing in $KK$- or $E$-theoretic proofs of the Baum-Connes conjecture for groups admitting a proper affine isometric action on a Hilbert space.

Here for each $t\in [1,\infty)$ we define 
$$
\gamma_t:\mathcal{S}\hat{\odot}\prod^U\C[\wgn]^{\Ga_n}\to \umk
$$
by the formula 
\begin{equation}\label{gammadef}
(\gamma_t (g\hat{\odot}\mathbf{T}))^{(n)}_{x,y}=T^{(n)}_{x,y}\hat{\otimes}g_{t^2}(B_{0,t}(x))
\end{equation}
Define 
$$
\gamma_t:\prod^{U,L}\C[\wgn]^{\Ga_n}\to \umlk
$$
by applying the above map pointwise.

\begin{gammalem}\label{gammalem}
The maps defined above extend to asymptotic morphisms
$$
\gamma_t:\mathcal{S}\hat{\otimes}\prod^{U,max}C^*(\wgn)^{\Ga_n}\to \umk
$$
and 
$$
\gamma_t:\mathcal{S}\hat{\otimes}\prod^{U,L,max}C^*(\wgn)^{\Ga_n}\to \umlk.
$$
\end{gammalem}

\begin{proof}
The proof is essentially the same as that showing $\beta_t$ gives rise to an asymptotic morphism in Lemma \ref{betalem} above.
\end{proof}

We denote the maps induced on $K$-theory by these asymptotic morphisms
$$
\gamma_*:K_*(\prod^{U,max}C^*(\wgn)^{\Ga_n})\to K_*(\umk)
$$
and 
$$
\gamma_*:K_*(\prod^{U,L,max}C^*(\wgn)^{\Ga_n})\to K_*(\umlk)
$$
respectively.

\begin{gammaiso}\label{gammiso}
The asymptotic morphisms $\gamma_t$ in Lemma \ref{gammalem} above induce isomorphisms on $K$-theory.
\end{gammaiso}

\begin{proof}
The following argument is adapted from the proof of the Baum-Connes conjecture for a-T-menable groups -- see for example \cite[proof of Theorem 3.11]{Higson:2004la}.  We will, as usual, only consider the map
$$
\gamma_*:K_*(\prod^{U,max}C^*(\wgn)^{\Ga_n})\to K_*(\umk);
$$
the case of the localization algebras is similar.

For each $n\in\N$ write the $\Ga_n$ action on $\V_n$ as 
$$
\sigma^{(n)}(g):v\mapsto \pi^{(n)}(g)v+b^{(n)}(g),
$$
where $\pi^{(n)}$ is the linear part of the action (and thus defines a group homomorphism from $\Ga_n$ into the linear isometries on $\V$), and $b^{(n)}:\Ga_n\to \V_n$ is a proper cocycle for this action.  Define for each $s\in[0,1]$ a new action by 
$$\sigma_s^{(n)}:v\mapsto \pi^{(n)}(g)v+sb^{(n)}(g)$$
(thus in some sense $\sigma^{(n)}_s$ forms a `homotopy' between the original action and its linear part).  Write $\umk_s$ for the $C^*$-algebra defined as $\umk$, but where the $\sigma^{(n)}_s$-action has been used to define the $\Ga_n$ action on $\mathcal{K}(\mathcal{L}^2_n)$ for each $n$.  Define now for each $s\in[0,1]$ an asymptotic morphism 
$$
\gamma_{s,t}:\prod^{U,max}C^*(P_R(\wgn))^\Ga_n\to \umk_s
$$
by the formula 
$$
(\gamma_t (g\hat{\odot}\mathbf{T}))^{(n)}_{x,y}=T^{(n)}_{x,y}\hat{\otimes}g_{t^2}(B_{0,t}(sf(x))),
$$
(the same argument as in Lemma \ref{gammalem} shows that these are all asymptotic morphisms).

The $C^*$-algebras $\umk_s$ then fit together into a continuous field over $[0,1]$, which we denote by 
$$
\umk_{[0,1]},
$$
and which is equipped with an asymptotic morphism 
$$
\gamma_{[0,1],t}:\prod^{U,max}C^*(P_R(\wgn))^{\Ga_n}\to \umk_{[0,1]}
$$
defined by piecing together the various $\gamma_{s,t}$s.  Note moreover that $\umk_{[0,1]}$ is equipped with a family of evaluation maps 
$$\epsilon_s:\umk_{[0,1]}\to\umk_s$$
that fit into a commutative diagram
$$
\xymatrix{ \prod^{U,max}C^*(P_R(\wgn))^{\Ga_n} \ar@{=}[d] \ar[r]^(.42){\gamma_{[0,1],t}} & \umk_{[0,1]} \ar[d]^{\epsilon_s} \\ \prod^{U,max}C^*(P_R(\wgn))^{\Ga_n} \ar[r]^(.42){\gamma_{s,t}}&  \umk_s };
$$
as the maps $\epsilon_s$ all induce isomorphisms on $K$-theory, to show that $\gamma_t=\gamma_{1,t}$ induces an isomorphism on $K$-theory, it suffices to show that the map $\gamma_{0,t}$ does.

Let then $p^{(n)}$ be the projection onto the kernel of $B_{0,t}(0)$ in $\mathcal{K}(\mathcal{L}^2_n)$.  Then $\gamma_{0,t}$, which is given by the formula
$$
(\gamma_{0,t} (g\hat{\odot}\mathbf{T}))^{(n)}_{x,y}=T^{(n)}_{x,y}\hat{\otimes}g_{t^2}(B_{0,t}(0))
$$
is homotopic to the map defined by 
$$
(\gamma_{P,t}(g\hat{\odot} \mathbf{T}))^{(n)}_{x,y} \mapsto T^{(n)}_{x,y}\hat{\otimes}g(0) P
$$
via the homotopy
$$
(\gamma_{0,t}(s)g\hat{\otimes}\mathbf{T})^{(n)}_{x,y}:=\left\{\begin{array}{ll} T^{(n)}_{x,y}\hat{\otimes}g_t(s^{-1}B_{0,t}(0)) & s\in (0,1] \\ T^{(n)}_{x,y}\hat{\otimes}g(0)P & s=0 \end{array}\right.;
$$
this completes the proof.
\end{proof}

At this point, we have that the diagram in line \eqref{ddd} above exists; it is moreover immediate from the definitions that it is commutative.  The proof of Theorem \ref{dddthe}, and therefore also of Theorem \ref{maxthe} will thus be completed by the next lemma.

\begin{complem}\label{complem}
On the level of both the uniform products of Roe algebras, and of the uniform products of localization algebras, the composition of Dirac and Bott morphisms is the Gamma morphism; in symbols $\alpha_*\circ\beta_*=\gamma_*$.
\end{complem} 

\begin{proof}
Using \cite[Lemma A.2]{Willett:2010ca}, the product $\alpha_*\circ\beta_*$ is the map induced on $K$-theory by the `naive composition' of asymptotic morphisms $\alpha_t\circ\beta_t$.  This, however, is asymptotically equivalent to $\gamma_t$ by the argument of \cite[Proposition 4.2]{Higson:1999be} (cf.\ also \cite[proof of Proposition 7.7]{Yu:200ve}).
\end{proof}

\section{Geometric property (T)}\label{geotsec}

In this section we introduce \emph{geometric property (T)}, and show that the maximal coarse assembly map fails to be surjective for any space of graphs $X=\sqcup G_n$ with geometric property (T). Thus geometric property (T) forms a strong opposite to having girth tending to infinity.  There is a good analogy here with Kazhdan's property (T) for a group $\Ga$, which is an obstruction to surjectivity of the maximal Baum-Connes assembly map for $\Ga$; we also show that the space of graphs constructed from a sequence of finite quotients of a group $\Ga$ had geometric property (T) if and only if $\Ga$ itself has property (T).  The definition of geometric property (T) was suggested by work of Oyono-Oyono and the second author \cite{Oyono-Oyono:2009ua}. 

\begin{geot}\label{geot}
Let $X=\sqcup G_n$ be a space of graphs, and recall that 
$$\Delta:=\prod_n(\Delta_n\otimes q)\in\C[X]$$
is defined to be the direct product of the graph Laplacians on each $G_n$ tensored by some rank one projection $q\in\mathcal{K}$.
$X$ is said to have \emph{geometric property (T)} if the spectrum of $\Delta$ in $C^*_{max}(X)$ is contained in $\{0\}\cup[c,2]$ for some $c>0$.
\end{geot}
\noindent
Of course, if we assume that the spectrum of $\Delta$ as an element of $C^*(X)$ is contained in $\{0\}\sqcup[c,2]$ for some $c>0$, we are asserting precisely that $X$ is an expander.  Thus `being an expander' is weaker than `having geometric property (T)'.

One could also define geometric property (T) for general (possibly disconnected) graphs, and probably even for general metric spaces; it seems likely, moreover, that any reasonable definition would be a coarse invariant (although we have not checked this).  Geometric property (T) may have fairly different properties outside of the current context: for example, a connected graph with \emph{property A} (e.g.\ the Cayley graph of an exact group) has geometric property (T) if and only if it is not \emph{amenable} in the sense of \cite{Block:1992qp} (and therefore if and only if not amenable in the group-theoretic sense, if the Cayley graph of a group).  These issues seem to merit further investigation.

The following lemma gives some examples.

\begin{geotex}\label{geotex}
Say $\Gamma$ is a finitely generated infinite group and $(\Gamma_n)_{n\in\N}$ an infinite nested sequence of finite index normal subgroups such that $\cap\Ga_n=\{e\}$.  Then $X=\sqcup \Ga/\Ga_n$ has geometric property (T) if and only if $\Ga$ has property (T).
\end{geotex}

\begin{proof}
Fix a rank one projection $q\in\mathcal{K}$, and note that there is a map
\begin{equation}\label{alginj}
\C[\Ga]\to \C[X],~~~g\mapsto v_g\otimes q,
\end{equation}
where $v_g$ is the unitary element coming from the right action of $\Ga$ on each $\Ga/\Ga_n$.  Note that if $S\subseteq\Ga$ is the symmetric generating set with respect to which the graph structures on each $\Ga/\Ga_n$ are defined, then 
$$\Delta_\Ga:=I-\frac{1}{|S|}\sum_{s\in S}s\in\C[\Ga]$$
maps to $\Delta\in \C[X]$ under the $*$-homomorphism above. 

Now, it follows from \cite[Proposition 2.8]{Oyono-Oyono:2009ua} and the fact that $l^\infty(X)$ has a $\Ga$-invariant state that the map in line \eqref{alginj} above extends to an injective $*$-homomorphism 
$$
C^*_{max}(\Ga)\to C^*_{max}(X).
$$
Hence the spectrum of $\Delta_\Ga$ in $C^*_{max}(X)$ is the same as its spectrum in $C^*_{max}(\Ga)$; the latter is well-known to be contained in some $\{0\}\cup[c,2]$ if and only if $\Ga$ has property (T), however.
\end{proof}
\noindent
Thus geometric property (T) for spaces of graphs naturally extends the notion of `Margulis-type expander from a property (T) group'.  We currently do not know any other examples, but suspect such exist (see problem \ref{gtprob} at the end of this section).

\begin{geotthe}\label{geotthe}
Say $X=\sqcup G_n$ has geometric property (T).  Then the maximal coarse assembly map 
$$
\mu:\lim_{R\to\infty}K_*(P_R(X))\to K_*(C^*_{max}(X))
$$
is not surjective.
\end{geotthe}

\begin{proof}
Let $p=\lim_{t\to\infty}e^{-t\Delta}$ be the spectral projection associated to $0\in\text{spectrum}(\Delta)$ in $C^*_{max}(X)$; the limit exists by geometric property (T).  The result of \cite[Corollary 3.9]{Willett:2010ud} is that there is a $*$-homomorphism
$$
\phi:C^*_{max}(X)\to\frac{\prod_nC^*(\wgn)^{\Ga_n}}{\oplus_nC^*(\wgn)^{\Ga_n}}
$$
which is such that $\phi(p)=0$ by the proof of \cite[Lemma 5.6]{Willett:2010ud}.   Essentially the same proof as in \cite[Lemma 6.5]{Willett:2010ud} using the Atiyah $\Ga$-index theorem (\cite[Theorem 6.4]{Willett:2010ud}) shows that $[p]\in K_0(C^*(X))$ is not in the image of the maximal coarse assembly map.
\end{proof}

We explicitly note the following corollaries, the last two of which are purely geometric.  We do not, however, know of any way to prove them that does not use operator K-theory.

\begin{geotc}\label{geotc} 
\begin{verbatim}\end{verbatim}
\begin{enumerate}[(i)]
\item Bounded geometry expanders with girth tending to infinity do not have geometric property (T).
\item No bounded geometry expander with girth tending to infinity is coarsely equivalent to a Margulis-type expander constructed from a property (T) group.
\item No Margulis-type expander constructed from a property (T) group can be coarsely embedded in a group using the techniques of Gromov \cite{Gromov:2003gf} and Arzhantseva--Delzant \cite{Arzhantseva:2008bv} (at least not without somehow dropping the girth assumption in their work). \qed
\end{enumerate}
\end{geotc}

We conclude this section with some problems that seem interesting.

\begin{gtprob}\label{gtprob}
Find purely geometric conditions that guarantee geometric property (T).  Develop geometric property (T), or possibly some variant, for spaces other than unions of finite graphs.
\end{gtprob} 

A good geometric criterion for geometric property (T) (possibly modeled after work of \.{Z}uk on spectral criteria for property (T) itself \cite{Zuk:2003jf}) may also provide an approach to the following problem.  Note in this regard that `random graphs' seem to have `small girth' (see for example the remark on \cite[page 22]{Arzhantseva:2008bv} in this regard).

\begin{gtprob2}\label{gtprob2}
Is geometric property (T) `generic' among all spaces of graphs $X=\sqcup G_n$?
\end{gtprob2} 

Finally, part (iii) of Corollary \ref{geotc} makes the following question very natural.  

\begin{gtprob3}\label{gtprob3}
Can an expander with geometric property (T) be coarsely embedded in a countable discrete group?
\end{gtprob3}

\bibliography{Generalbib}

\end{document}